\documentclass[onefignum,onetabnum]{siamart220329}
\usepackage{graphics,graphicx,epstopdf,url}
\usepackage{arydshln}
\usepackage{color}
\usepackage{algorithmic}
\usepackage{url,cases,supertabular}
\usepackage{amssymb,mathtools}
\usepackage{enumerate,makecell}
\usepackage{amsfonts}
\usepackage{graphicx,epstopdf}
\usepackage{color,verbatim}
\usepackage{mathrsfs,subeqnarray,subfigure}
\usepackage{listings,array}
\usepackage{booktabs}
\usepackage{dashrule}
\usepackage{arydshln}
\usepackage{graphicx}
\usepackage{amsfonts}
\usepackage{mathrsfs}
\usepackage{graphicx}  \usepackage{epstopdf}
\usepackage{subfigure}  \usepackage{lineno}
\usepackage{multirow,bm}
\usepackage{color}
\usepackage{algorithm}
\usepackage{algorithmic}
\usepackage{longtable}
\usepackage{booktabs}
\usepackage{rotating}
\usepackage{enumitem}
\usepackage{lineno}
\usepackage{url}
\usepackage{pdflscape}
\usepackage{makecell}

\newcommand{\revise}[1]{{\color{black}#1}}

\numberwithin{equation}{section}

\newtheorem{Lemma}[theorem]{Lemma}
\newtheorem{Proposition}[theorem]{Proposition}

\newtheorem{Assumption}[theorem]{Assumption}

\newcommand{\ba}{\begin{array}}
\newcommand{\ea}{\end{array}}

\newcommand{\bit}{\begin{itemize}}
\newcommand{\eit}{\end{itemize}}
\newcommand{\be}{\begin{equation}}
\newcommand{\ee}{\end{equation}}
\newcommand{\bea}{\begin{eqnarray}}
\newcommand{\eea}{\end{eqnarray}}

\newcommand{\st}{\mathrm{s.t.}}

\newcommand{\argmin}{\mathop{\mathrm{arg\,min}}}

\newcommand{\diag}{\text{diag}}

\newcommand{\Rmn}[1]{\uppercase\expandafter{\romannumeral#1}}
\renewcommand{\theenumi}{\roman{enumi}}

\numberwithin{equation}{section}

\newcommand{\R}{\mathbb{R}}

\numberwithin{theorem}{section}
\newcommand{\iprod}[2]{\left \langle #1, #2 \right \rangle }

\usepackage{lipsum}
\usepackage{clrscode}
\usepackage{appendix}
\usepackage{bm}
\usepackage{booktabs}
\usepackage{url}
\usepackage{multirow}
\usepackage{textcomp}
\usepackage{amsmath}
\usepackage{amsfonts}
\usepackage{amssymb}
\usepackage{mathrsfs}
\usepackage{hyperref}
\usepackage{graphicx,graphics,subfigure}
\usepackage{hyperref,url}
\usepackage{epsf,epstopdf}
\usepackage{algorithm}
\usepackage{algorithmic}
\usepackage{bbm}
\usepackage{dsfont}
\usepackage{indentfirst}
\usepackage{pdfpages}
\usepackage{pdflscape}
\usepackage{longtable}
\ifpdf
  \DeclareGraphicsExtensions{.eps,.pdf,.png,.jpg}
\else
  \DeclareGraphicsExtensions{.eps}
\fi

\headers{A Projected SSN FOR NONCONVEX AND NONSMOOTH PROGRAMS}{Jiang Hu, Kangkang Deng, Jiayuan Wu, and Quanzheng Li}

\title{A projected semismooth Newton method for a class of nonconvex composite programs \revise{with strong prox-regularity}
}

\author{Jiang Hu\thanks{Massachusetts General Hospital and Harvard Medical School, Harvard University, Boston, MA 02114
(\email{hujiangopt@gmail.com}).}
  \and  Kangkang Deng\thanks{Corresponding author. Beijing International Center for Mathematical Research, Peking University, Beijing 100871,
China (\email{freedeng1208@gmail.com}).}
 \and Jiayuan Wu\thanks{College of Engineering, Peking University, Beijing, China
	(\email{1901110043@pku.edu.cn}).}
 \and Quanzheng Li\thanks{Massachusetts General Hospital and Harvard Medical School, Harvard University, Boston, MA 02114
(\email{li.quanzheng@mgh.harvard.edu}).}
}

\usepackage{amsopn}

\ifpdf
\hypersetup{
  pdftitle={A Projected SSN for Nonconvex and Nonsmooth Programs},
  pdfauthor={}
}
\fi

\begin{document}

\maketitle

\begin{abstract}
This paper aims to develop a Newton-type method to solve a class of nonconvex composite programs. In particular, the nonsmooth part is possibly nonconvex. To tackle the nonconvexity, we develop a notion of strong prox-regularity which is related to the singleton property, Lipschitz continuity, and monotonicity of the associated proximal operator, and we verify it in various classes of functions, including weakly convex functions, indicator functions of proximally smooth sets, \revise{and two specific sphere-related nonconvex nonsmooth functions. In this case, the problem class we are concerned with covers smooth optimization problems on manifold and certain composite optimization problems on manifold. For the latter, the proposed algorithm is the first second-order type method.} Combining with the semismoothness of the proximal operator, we design a projected semismooth Newton method to find a root of the natural residual induced by the proximal gradient method. Since the corresponding natural residual may not be globally monotone,  an extra projection is added on the usual semismooth Newton step and  new criteria 
are proposed for the switching between the projected semismooth Newton step and the proximal step. The global convergence is then established under the strong prox-regularity. Based on the BD regularity condition, we establish local superlinear convergence.
Numerical experiments demonstrate the effectiveness of our proposed method compared with state-of-the-art ones. 
\end{abstract}

% REQUIRED

\section{Introduction}
The nonconvex composite minimization problem has attracted lots of attention in signal processing, statistics, and machine learning. The formulation we are concerned with is:
\be \label{prob} \min_{x \in \R^{n}} \;\; \varphi(x) := f(x) + h(x), \ee
where $f: \R^{n} \rightarrow \R$ is twice differentiable and possibly nonconvex, $h: \R^n \rightarrow (-\infty, \infty]$ is a proper, closed, and extended-value function.  Note that $h$ can be nonsmooth and nonconvex. In this paper, we consider a class of nonsmooth and nonconvex functions $h$ satisfying the following strong prox-regularity.
\begin{definition}[strong prox-regularity] \label{assum:h} \revise{We call a proper, closed, and extended function $h:\R^n \rightarrow \bar{\R}$ is strongly prox-regular if the} proximal operator ${\rm prox}_{th}(x):= \argmin_{u} th(u) + \frac{1}{2}\|x - u\|^2$ is single-valued, Lipschitz continuous, and monotone over a set $\{ x + tv: x \in \mathcal{C} \subset \R^n, v \in \R^n ~\mathrm{with} ~\|v\| = 1, 0 \leq t \leq \gamma \}$ with a compact set $\mathcal{C} \supset {\rm dom}(h)$, a positive constant \revise{$\gamma$, and a norm function $\| \cdot \|$}.
\end{definition}

We call the above definition strong prox-regularity due to the uniform $\gamma$ for all $x \in \mathcal{C}$, which can be seen as an enhanced version of the prox-regularity \cite[Definition 13.27, Proposition 13.37]{rockafellar2009variational}.
Note that the strong prox-regularity holds for any compact $\mathcal{C} \subset \R^n$ and $\gamma >0$ if $h$ is convex \cite{moreau1965proximite}. Here, we present some classes of nonconvex functions satisfying Definition \ref{assum:h}.
% For nonconvex $h$, Assumption \ref{assum:h} is satified in the following cases.
\begin{itemize}
    \item[{\rm (i)}] $h$ is weakly convex. A function is called weakly convex with modulus $\rho >0$ if $h(x) + \frac{\rho}{2}\|x\|^2$ is convex. By using the same idea for the convex functions, one can verify that $\mathrm{prox}_{t h}$ is single-valued, Lipschitz continuous, and monotone when $t < \frac{1}{\rho}$. Optimization with weakly convex objective functions has been considered in \cite{davis2019stochastic}.
    \item[{\rm (ii)}] $h$ is the indicator function of a proximally smooth set \cite{clarke1995proximal}. For a set $\mathcal{X} \subset \R^n$,	define its closed $r$-neighborhoods
	\be \label{eq:prox-smooth-nei}	\mathcal{X}(r):=\left\{u \in \R^n: d_{\mathcal{X}}(u) \leq r\right\}, \;\; \text{with} \;\; d_{\mathcal{X}}(u):=\inf \{\|u-x\|: x \in \mathcal{X}\}. \ee
	We say that $\mathcal{X}$ is $r$-proximally smooth if the
	nearest-point projection $\mathrm{proj}_{\mathcal{X}}$ is single-valued on $\mathcal{X}(r)$. In addition, the proximal operator (which is the same as the projection operator onto $\mathcal{X}$) is Lipschitz continuous and monotone \cite[Theorem 4.8]{clarke1995proximal} on $\mathcal{X}(r)$. Note that the projection operator onto a smooth and compact manifold embedded in Euclidean space is a smooth mapping on a neighborhood of the manifold \cite{foote1984regularity}. It is also worth mentioning that the Stiefel manifold is 1-proximally smooth \cite[Proposition 1]{balashov2022error}.
\end{itemize}

As shown above, optimization with weakly convex regularizers or constraints of the proximally smooth set can be fitted into \eqref{prob}. The strong prox-regularity serves as a general concept to put different problem classes together and allows us to derive a uniform algorithmic design and theoretic analysis. Since the proximal operator is single-valued, Lipschitz continuous, and monotone on a compact set, one can further explore the differentiability and design second-order type algorithms to obtain the algorithmic speedup and fast convergence rate guarantee.

It has been shown in \cite{bohm2021variable} that two popular nonsmooth nonconvex regularizers, the minimax concave penalty \cite{zhang2010nearly} and the smoothly clipped absolute deviation \cite{fan1997comments}, are weakly convex. Since any smooth manifold is proximally smooth, the manifold optimization problems \cite{absil2009optimization,hu2020brief,boumal2020introduction} take the form \eqref{prob}. Besides, we are also motivated by the following applications, where $h$ is from the oblique manifold and a simple $\ell_1$ norm or the constraint of nonnegativity. \revise{Let us note that such $h$ is not weakly convex or the indicator function of a smooth manifold.}

\subsection{Motivating examples} \label{sec:motivate-exam}

\subsubsection*{Example 1. Sparse PCA on oblique manifold}
	In \cite{huang2021riemannian}, the authors consider the following formulation of sparse PCA:
	\be \label{prob:spca} \min_{X \in {\rm Ob}(n,p)} \;\; \|X^\top A^\top A X - D^2 \|_F^2 + \lambda \|X\|_1,  \ee
	where ${\rm Ob}(n,p) = \{ X \in \R^{n \times p}:\diag(X^\top X) = 1_{p} \}$ with $\diag(B)$ being a vector consisting of the diagonal entries of $B$ and $1_{p} \in \R^n$ of all elements 1, $D$ is a diagonal matrix whose diagonal entries are the first $p$ largest singular values of $A$,  $\|\cdot \|_F$ denotes the matrix Frobenius norm, $\|X\|_1:= \sum_{i=1}^n\sum_{j=1}^p |X_{ij}|$, and $\lambda > 0$ is a parameter to control the sparsity. Problem \eqref{prob:spca} takes the form \eqref{prob} by letting
	\begin{equation}\label{eq:l1oblique}
	   h(X) = \lambda \|X\|_1 + \delta_{{\rm{Ob}(n,p)}}(X),
	\end{equation}
	 where $\delta_{\mathcal{C}}(\cdot)$ denotes the indicator function of the set $\mathcal{C}$. Utilizing the separable structure and the results by \cite{xiao2021geometric}, the $i$-th column of ${\rm prox}_{th}(X)$, denoted by $({\rm prox}_{th}(X))_i$, is
    \[ ({\rm prox}_{th}(X))_i =
     \begin{cases}
    	(\underbrace{0, \ldots, 0}_{j-1}, {\rm sign}(X_{ij}), \underbrace{0, \ldots, 0}_{n-j})^\top, \;\; & {\rm if} \;\; w \geq 0, \\
    	-w_i^-/\|w_i^-\|_2\cdot {\rm sign}(X_i), \;\; & {\rm otherwise},
    \end{cases} \]
where $w_i = \lambda t - |X_i|$, $X_i$ is the $i$-th column of $X$, $w_i^- = \min(w_i, 0)$, $\mathrm{sign}(a)$ returns $1$ if $a \geq 0$ and $-1$ otherwise, and $j = \argmin_{1\leq k \leq n} w_i(k)$. Note that $\mathrm{prox}_{th}$ is not unique for all $X \in \R^{n\times p}$ and $t > 0$. We will give the specific $\mathcal{C}$ and $\gamma$ such that ${\rm prox}_{th}$ is strongly prox-regular later in Section \ref{sec:verifty-prox-regularity}.

	\subsubsection*{Example 2. Nonnegative PCA on oblique manifold}
	If the nonnegativity of the principal components is required, we have the following nonnegative PCA model
	\be \label{prob:npca} \min_{X \in {\rm Ob}^{+}(n,p)} \;\; \|X^\top A^\top A X - D^2 \|_F^2, \ee
	where ${\rm Ob}^{+}(n,p) := {\rm Ob}(n,p) \cap \{X \in \R^{n\times p}: X_{ij} \geq 0 \}$ and $D$ is defined as in \eqref{prob:spca}. Note that a more general formulation with smooth objective function over ${\rm Ob}^{+}(n,p)$ has been considered in \cite{jiang2019exact}. 
	Problem \eqref{prob:npca} falls into \eqref{prob} by letting
	\begin{equation}\label{eq:negoblique}
	    h(X) = \delta_{{{\rm{Ob}}^+(n,p)}}(X)
	\end{equation}
	 is the indicator function of $
	\mathrm{Ob}^+(n,p)$.
% 	It follows from the definition that
% 	\[ \begin{aligned}
% 		{\rm prox}_{th}(X)  = \argmin_{Y\in {\rm Ob}^+(n,p)} \;\; \frac{1}{2} \|Y - X \|_F^2
% 		= {\rm Proj}_{{\rm Ob}^+(n,p)}(X).
% 	\end{aligned} \]
Due to the separable structure, the $i$-th column of ${\rm prox}_{th}(X)$, denoted by $({\rm prox}_{th}(X))_i$, is
\[ ({\rm prox}_{th}(X))_i =\begin{cases}
	(\underbrace{0, \ldots, 0}_{j-1}, 1, \underbrace{0, \ldots, 0}_{n-j}), & \;\; {\rm if} \; \max(X_i) \leq 0, \\
	X_i^+ /\|X_i^+\|_2, & \;\; {\rm otherwise},
\end{cases}   \]
where $j = \argmin_{1\leq k\leq n} X_{ik}$ in the first case, $X_i^+ = \max(X_i, 0)$, and $X_i$ is the $i$-th column of $X$. Note that this projection is not unique for all $X \in \R^{n\times p}$, e.g., $X = 0$. We will show its strong prox-regularity later in Section \ref{sec:verifty-prox-regularity}.
\subsubsection*{Example 3. Sparse least square regression with probabilistic simplex constraint}
	
	The authors of \cite{xiao2021geometric,li2021simplex} consider the spherical constrained formulation of the following optimization problems:
	\be \label{prob:slr} \begin{aligned}
		\min_{y \in \R^n} \;\; \frac{1}{2} \|Ay - b\|_2^2, \;\; \st \;\; y \in \Delta_n,
	\end{aligned} \ee
	where $\Delta_n = \{y\in \R^n: y \geq 0, 1_n^\top y = 1 \}$, $A \in \R^{m \times n}$, and $b \in \R^m$. By decomposing $y = x \odot x$, it holds that
	\[ y \in \Delta_n \;\; \Longleftrightarrow \;\; x \in {\rm Ob}(n,1). \]
	Adding a sparsity constraint on $x$ leads to the following optimization problem
	\be \label{prob:slr1} \begin{aligned}
		\min_{x \in \R^n} \;\; \frac{1}{2} \|A(x\odot x) - b\|_2^2 + \lambda \|x\|_1, \;\; \st \;\; x \in {\rm Ob}(n,1).
	\end{aligned} \ee
	By taking $h(x) = \lambda\|x\|_1 + \delta_{{\rm Ob}(n,1)}$, problem \eqref{prob:slr1} has the form \eqref{prob}. Due to the separable structure of the proximal operator of \eqref{eq:l1oblique}, the strong prox-regularity of $h$ here is similar to that of \eqref{eq:l1oblique}.
	
\subsection{Literature review}

The composite optimization problem arises from various applications, such as signal processing, statistics, and machine learning. When $h$ is convex, extensive first-order methods are designed, such as the proximal gradients and its Nesterov's accelerated versions, the alternating direction methods of multipliers, etc. We refer to \cite{boyd2011distributed,beck2017first} for details. For faster convergence, second-order methods, such as proximal Newton methods \cite{lee2014proximal,kanzow2021globalized} and semismooth Newton methods \cite{mifflin1977semismooth,qi1993nonsmooth,qi1999survey,byrd2016family,milzarek2014semismooth,zhao2010newton,xiao2018regularized,li2018highly,li2018semismooth} are also developed for the nonsmooth problem \eqref{prob}. If $h$ is nonconvex, the proximal gradient methods are developed for $\ell_{1/2}$ norm in \cite{xu2012l_} and more nonconvex regularizers \cite{gong2013general,yang2017proximal}. The global convergence is established by utilizing the smoothness of $f$ and the explicit solution of the proximal subproblem.

In the case of $h$ being weakly convex, subgradient-type methods \cite{davis2019stochastic,davis2018subgradient} and proximal point-type method \cite{drusvyatskiy2017proximal} yield lower complexity bound.
% several algorithms are proposed for improving the complexity bound, including subgradient-type methods \cite{davis2019stochastic,davis2018subgradient} and proximal point-type methods \cite{drusvyatskiy2017proximal}. 
Optimization with prox-regular functions has recently attracted much attention. The authors \cite{themelis2018forward} propose a gradient-type method to solve the forward-backward envelope of $\varphi$. This can be seen as a variable-metric first-order method.
Since the Moreau envelope of a prox-regular function is continuously differentiable, a nonsmooth Newton method is designed to solve the gradient system of the Moreau envelope in \cite{khanh2020generalized,khanh2021globally}. Note that the indicator function of a proximally smooth set is prox-regular \cite{clarke1995proximal}, the authors of \cite{balashov2022error} developed a generalized Newton method to fixed point equation induced by the projected gradient method.

In the case of $h$ being the indicator function of a Riemannian manifold, the efficient Riemannian algorithms have been extensively studied in the last decades \cite{absil2009optimization,wen2013feasible,hu2020brief,boumal2020introduction}.
When $h$ takes the form \eqref{eq:l1oblique}, the manifold proximal gradient methods \cite{chen2018proximal,huang2021riemannian} are designed. These approaches only use first-order information and do not have superlinear convergence. In addition, manifold augmented Lagrangian methods are also proposed in works \cite{deng2022manifold,zhou2021semi}, in which the subproblem is solved by the first-order method or second-order method. When it comes to the case of \eqref{eq:negoblique}, a second-order type method is proposed in the recent work \cite{jiang2019exact}. While in their subproblems, only the second-order information of the smooth part is explored.
\subsection{Our contributions}In this paper, we propose a projected semismooth Newton method to deal with a class of nonsmooth and nonconvex composite programs. In particular, the nonsmooth part is nonconvex but satisfies the proposed strong prox-regularity properties.  
% By taking advantage of the generalized Jacobian of the natural residual induced by the proximal gradient method, the proposed second-order method  enjoys superlinear convergence. 
Our main contributions are as follows:
\begin{itemize}
    \item We introduce the concept of strong prox-regularity. Different from the classic prox-regularity, the strong prox-regularity enjoys some kind of uniform proximal regularity around a compact region containing all feasible points. A crucial property is that the proximal operator of a strongly prox-regular function locally behaves like that of convex functions. With the strong prox-regularity, the stationary condition can be reformulated as a single-valued residual mapping which is Lipschitz continuous and monotone on the compact region. 
    % If the residual mapping is also semismooth, we are able to design efficient semismooth Newton methods. 
    We present several classes of functions satisfying both the strong prox-regularity condition, including weakly convex functions and indicator functions of proximally smooth sets (including manifold constraints). \revise{In particular, two specific sphere-related nonsmooth and nonconvex functions, which are not weakly convex or indicator functions of a smooth manifold, are verified to satisfy the strong prox-regularity.}
    \item \revise{As shown in Section \ref{sec:motivate-exam}, two sphere-related nonsmooth and nonconvex functions result in composite optimization problems on manifold. In this paper, we propose the first second-order type method to solve this kind of problem, which outperforms state-of-the-art first-order methods \cite{chen2018proximal,huang2021riemannian}. It is worth mentioning that first-order methods \cite{chen2018proximal,huang2021riemannian} fails in solving the nonnegative PCA on the oblique manifold due to their dependence on the Lipschitz continuity of the nonsmooth part.}
    \item By introducing the strong prox-regularity condition and semismoothness, we design a residual-based projected semismooth Newton method to solve the nonconvex composite optimization problem \eqref{prob}. To tackle the nonconvexity, we add an extra projection on the usual semismooth Newton step and switch to the proximal gradient step if two proposed inexact conditions are not satisfied. Compared with the Moreau-envelope based approaches \cite{khanh2020generalized,khanh2021globally}, we decouple the composite structures and design a second-order method by utilizing the second-order derivative of the smooth part and the generalized Jacobian of the proximal operator of $h$. 
    \item The global convergence of the proposed projected semismooth Newton method is presented. Other than the strong prox-regularity condition and the semismoothness, the assumptions are standard and can be achieved by various applications including our motivating examples. We prove the switching conditions are locally satisfied, which allows the local transition to the projected semismooth Newton step. By assuming the BD-regularity condition, we show the local superlinear convergence.
    Numerical experiments on various applications demonstrate the efficiency over state-of-the-art ones.
\end{itemize}

\subsection{Notation}Throughout, we consider the Euclidean space $\mathbb{R}^n$, equipped with an inner product $\left<\cdot,\cdot\right>$ and the induced norm $\|x\| = \sqrt{\left<x,x\right>}$. Given a matrix $A$, we use $\|A\|_F$ to denote its Frobenius norm, $\|A\|_1:=\sum_{ij}|A_{ij}|$ to denote its $\ell_1$ norm. For a vector $x$, we use $\|x\|_2$ and $\|x\|_1$ to denote its Euclidean norm and $\ell_1$ norm, respectively. The indicator function of $\mathcal{C}$, denoted $\delta_{\mathcal{C}}$, is defined to be zero on $\mathcal{C}$ and $+\infty$ otherwise. The symbol $\mathbb{B}$ will denote the closed unit ball in $\mathbb{R}^n$, while $\mathbb{B}(x,\epsilon)$ will stand for the closed ball of the radius of $\epsilon>0$ centered at $x$.

\subsection{Organization}The outline of this paper is as follows. In Section \ref{sec:preli}, we present the preliminaries on the subdifferential, concepts of stationarity, and semismoothness. Various nonconvex and nonsmooth functions satisfying the strong prox-regularity and semismoothness are demonstrated in Section \ref{sec:verifty-prox-regularity}. Then, we propose a projected semismooth Newton method in Section \ref{sec:ssn}. The corresponding convergence analysis of the proposed method is provided in Section \ref{sec:conver}. We illustrate the efficiency of our proposed method by several numerical experiments in Section \ref{sec:num}. Finally, a brief conclusion is given in Section \ref{sec:con}.

\section{Preliminaries}\label{sec:preli}
In this section, we first review some basic notations of subdifferential and give the definition of the prox-regular function. We also introduce several concepts of stationarity and present the definition of semismoothness.
\subsection{Subdifferential and prox-regular functions}
Let $\varphi: \mathbb{R}^{n} \rightarrow(-\infty,+\infty]$ be a proper, lower semicontinuous, and extended real-valued function. The domain of $\varphi$ is defined as $\operatorname{dom}(\varphi)=\{{x} \in$ $\left.\mathbb{R}^{n}: \varphi({x})<+\infty\right\}$. A vector ${v} \in \mathbb{R}^{n}$ is said to be a Fr\'{e}chet subgradient of $\varphi$ at ${x} \in \operatorname{dom}(\varphi)$ if
\begin{equation}\label{eq:subdiff}
   \liminf _{{y} \rightarrow {x}, \atop {y} \neq {x}} \frac{\varphi({y})-\varphi({x})-\langle{v}, {y}-{x}\rangle}{\|{y}-{x}\|} \geq 0 .
\end{equation}
The set of vectors ${v} \in \mathbb{R}^{p}$ satisfying \eqref{eq:subdiff} is called the Fr\'{e}chet subdifferential of $\varphi$ at ${x} \in \operatorname{dom}(\varphi)$ and denoted by $\widehat{\partial} \varphi({x})$. The limiting subdifferential, or simply the subdifferential, of $\varphi$ at ${x} \in$ $\operatorname{dom}(\varphi)$ is defined as
$$
\partial \varphi({x})=\left\{{v} \in \mathbb{R}^{n}: \exists {x}^{k} \rightarrow {x}, {v}^{k} \rightarrow {v} \text { with } \varphi\left({x}^{k}\right) \rightarrow \varphi({x}), {v}^{k} \in \widehat{\partial} \varphi \left({x}^{k}\right)\right\} .
$$
By convention, if ${x} \notin \operatorname{dom}(\varphi)$, then $\partial \varphi({x})=\emptyset .$ The domain of $\partial \varphi$ is defined as $\operatorname{dom}(\partial \varphi)=$ $\left\{{x} \in \mathbb{R}^{n}: \partial \varphi({x}) \neq \emptyset\right\} .$ For the indicator function $\delta_{\mathcal{S}}: \mathbb{R}^{n} \rightarrow\{0,+\infty\}$ associated with the non-empty closed set $\mathcal{S} \subseteq \mathbb{R}^{n}$, we have
$$
\widehat{\partial} \delta_{\mathcal{S}}({x})=\left\{{v} \in \mathbb{R}^{n}: \limsup _{{y} \rightarrow {x}, {y} \in \mathcal{S}} \frac{\langle{v}, {y}-{x}\rangle}{\|{y}-{x}\|} \leq 0\right\} \quad \text { and } \quad \partial \delta_{\mathcal{S}}({x})=\mathcal{N}_{\mathcal{S}}({x})
$$
for any ${x} \in \mathcal{S}$, where $\mathcal{N}_{\mathcal{S}}({x})$ is the normal cone to $\mathcal{S}$ at ${x}$.

The function $\varphi$ is prox-bounded \cite[Definition 1.23]{rockafellar2009variational} if there exists $\lambda > 0$ such that $e_{\lambda} \varphi(x) := \inf_{y} \{ \varphi(y) + \frac{1}{2\lambda}\|y-x\|^2 \} > -\infty$ for some $x \in \R^n$. The supremum of the set of all such $\lambda$ is the threshold $\lambda_{\varphi}$ of prox-boundedness for $\varphi$. The function $\varphi$ is prox-regular \cite[Definition 13.27]{rockafellar2009variational} at $\bar{x}$ for $\bar{v}$ if $\varphi$ is finite and locally lower semicontinuous at $\bar{x}$ with $\bar{v} \in \partial \varphi(\bar{x})$, and there exist $\varepsilon > 0$ and $\rho \geq 0$ such that
\be  \label{eq:prox-regular} \varphi(x') \geq \varphi (x) + \iprod{v}{x'-x} - \frac{\rho}{2}\|x'-x\|^2, \;\forall x' \in \mathbb{B}(\bar{x},\varepsilon), \ee
when $v \in \partial \varphi(x), \; \|v-\bar{v}\| < \varepsilon, \|x - \bar{x}\| < \varepsilon, \varphi(x) < \varphi(\bar{x}) + \varepsilon$. If the above inequality holds for all $\bar{v} \in \partial \varphi(\bar{x})$, $\varphi$ is said to be prox-regular at $\bar{x}$. Note that the inequality \eqref{eq:prox-regular} holds for all $x'\in {\rm dom}(\varphi)$ and $v \in \partial \varphi(x)$ with a uniform $\rho$ if $\varphi$ is weakly convex. It follows from \cite[Exercise 13.35]{rockafellar2009variational} that the summation of a smooth function and a prox-regular function is prox-regular as well.
% Hence, the objective function of \eqref{prob} is prox-regular.

For prox-regular functions, we have the following fact.
\begin{Proposition}(\cite[Proposition 13.37]{rockafellar2009variational}, \cite[Lemma 6.3]{khanh2020generalized}) \label{prop:Lip-prox-map-weakly-convex}
Let $\varphi: \R^n \rightarrow \bar{\R}$ be proper, lower semicontinuous, and prox-bounded with threshold $\lambda_{\varphi}$. Suppose $\varphi$ is finite and prox-regular at $\bar{x}$ for $\bar{v} \in \partial \varphi(\bar{x})$. Then for any sufficiently small $\gamma \in (0,\lambda_{\varphi})$, the proximal mapping ${\rm prox}_{\gamma \varphi}$ is single-valued, monotone, and Lipschitz continuous around $\bar{x} + \gamma \bar{v}$ and satisfies the condition ${\rm prox}_{\lambda \varphi}(\bar{x}+\gamma \bar{v})=\bar{x}$.
\end{Proposition}

Our proposed prox-regularity condition is a stronger version of the well-known prox-regularity condition in optimization theory. Specifically, our condition requires the proximal operator to be single-valued, monotone, and Lipschitz continuous for a compact region $\mathcal{C}$ with a uniform $\gamma$. As shown later, the uniformity of $\gamma$ plays a critical role in determining the lower bound of step sizes in algorithmic design.

\subsection{Concepts of stationarity and their relationship}
There are two definitions of stationarities based on the subdifferential and the proximal gradient iteration.
\begin{itemize}
    \item Critical point: $x$ is a critical point if
    \be \label{eq:crit} 0 \in \partial \varphi(x) = \nabla f(x) + \partial h(x). \ee
    % where $\partial h(x)$ is the limiting subdifferential of $h$ at $x$.
    \item Fixed point of the proximal mapping:
    \be \label{eq:fixed-point} x \in \mathrm{prox}_{t h}(x - t \nabla f(x)), \ee
    where $t > 0$.
\end{itemize}
It follows from the definition of $\mathrm{prox}_{th}$ that any point $x$ satisfying \eqref{eq:fixed-point} yields $0 \in \nabla f(x) + \partial h(x),$
which implies $x$ is also a critical point. Inversely, a critical point may not satisfy \eqref{eq:fixed-point} due to the nonconvexity of $h$. Therefore, equation \eqref{eq:fixed-point} defines a stronger stationary point than \eqref{eq:crit}.

\subsection{Semismoothness}
By the Rademacher's theorem, a locally Lipschitz operator is almost everywhere differentiable. For a locally Lipschitz $F$, denote by $D_F$ the set of the differential points of $F$. The $B$-subdifferential at $x$ is defined as
$$
\partial_{B} F(x):=\left\{\lim _{k \rightarrow \infty} J\left(x^{k}\right) \mid x^{k} \in D_{F}, x^{k} \rightarrow x\right\},
$$
where $J(x)$ represents the Jacobian of $F$ at the differentiable point $x$. Obviously, $\partial_{B} F(x)$ may not be a singleton. The Clarke subdifferential $\partial_C F(x)$ is defined as
\[ \partial_C F(x) = \mathrm{conv}\left(\partial_{B} F(x)\right),\]
where $\mathrm{conv}(A)$ represents the closed convex hull of $A$. A locally Lipschitz continuous operator $F$ is called semismooth at $x$ if
\begin{itemize}
    \item $F$ is directional differentiable at $x$.
    \item For all $d$ and $J \in \partial_C F(x+d)$, it holds that
\[ \| F(x+d) -  F(x) - Jd \| = o(\|d\|), \;\; d \rightarrow 0. \]
\end{itemize}
We say $F$ is semismooth if $F$ is semismooth for any $x \in \R^n$.  If $f$ is twice continuously differentiable and $\mathrm{prox}_{th}$ is single-valued, Lipschitz continuous, and semismooth with respect to its B-subgradient $D(x)$, one can verify that if $I - t\nabla^2 f(x)$ is nonsingular, the operator $F(x):= \mathrm{prox}_{th}(x - t \nabla f(x)) - x$ is semismooth with respect to 
\be \label{eq:jacob-f} M(x):=\{ I - D(I - t\nabla^2 f(x)) : D \in D(x) \} \ee
by using the definition of semismoothness \cite{chan2008constraint}.

\section{Semismooth and strongly prox-regular functions} \label{sec:verifty-prox-regularity}
Let us verify the semismoothness and the strongly prox-regularity condition for some typical nonconvex nonsmooth functions $h$.
\subsection{Weakly convex function}
Following \cite{moreau1965proximite}, one can verify that the strong prox-regularity holds for $\rho$-weakly convex functions if $t \leq 1/\rho$. The semismoothness of the proximal operator of a weakly convex function generally does not hold, which happens in the convex case as well. 
% From the fact that all the convex functions are weakly convex with modulus $\rho = 0$, the strong prox-regularity follows directly from the results for convex functions \cite{moreau1965proximite}.
While two popular nonconvex regularizers for reducing bias are the minimax concave penalty (MCP) \cite{zhang2010nearly} and the smoothly clipped absolute deviation \cite{fan1997comments}, the semismoothness is satisfied. Specifically, the MCP is defined as
\[ h_{\lambda, \theta}(x):= \begin{cases}\lambda|x|-\frac{x^{2}}{2 \theta}, & |x| \leq \theta \lambda, \\ \frac{\theta \lambda^{2}}{2}, & \text { otherwise, }\end{cases} \]
where $\lambda$ and $\theta$ are two positive parameters. It is weakly convex with modulus $\rho = \theta^{-1}$. If $t < \theta$, the closed-form expression of the proximal operator is
\[ \operatorname{prox}_{t h}(x)= \begin{cases}0, & |x|<t  \lambda, \\ \frac{x-\lambda t \operatorname{sign}(x)}{1-(t / \theta)}, 
& t \lambda \leq|x| \leq \theta \lambda, \\ x, 
& |x|>\theta \lambda. \end{cases}
\]
The semismoothness property of the MCP regularizer is presented in \cite{shi2019semismooth}. Analogously, one can also verify the weak convexity of the SCAD regularizer and the semismoothness of its proximal operator. We refer to \cite{bohm2021variable} and \cite{shi2019semismooth} for the details. Numerical results in \cite{shi2019semismooth} exhibit the efficiency of semismooth Newton methods.   

\subsection{Smooth and compact embedded manifold}
Since any smooth manifold is a proximally smooth set, there exists a neighborhood $\mathcal{X}(r)$ of the form \eqref{eq:prox-smooth-nei} such that the projection is single-valued, Lipschitz continuous, and monotone \cite[Theorem 4.8]{clarke1995proximal}. On the other hand, the projection onto smooth and compact embedded manifold is also a smooth mapping \cite{foote1984regularity} on $\mathcal{X}(r)$. Putting them together, we conclude that the indicator function is strongly prox-regular and the corresponding projection operator is smooth over $\mathcal{X}(r)$. For a special sphere-constrained smooth optimization problem, the Bose-Einstein condensates, we will show the numerical superiority of our proposed method using strong prox-regularity and semismoothness. For general smooth optimization problems with orthogonal constraints, we refer the reader to \cite{gawlik2017iterative} for the calculations of the generalized Jacobian of the polar decomposition.

\subsection{Two specific oblique manifold related nonconvex functions}
We shall show that the nonconvex and nonsmooth functions \eqref{eq:l1oblique} and \eqref{eq:negoblique} satisfy the strong prox-regularity and semismoothness.
\begin{lemma} \label{lem:strong-prox}
The functions $h$ defined in both \eqref{eq:l1oblique} and \eqref{eq:negoblique} are strongly prox-regular and their proximal operators are semismooth. Specifically,
\begin{itemize} 
    \item[{\rm (i)}] Let $\mathcal{C}_1 = {\rm Ob}(n,p)$, $\|V\|_{2,\infty} := \max_{i=1,2,\ldots,p} \|V_i\|_2$, and $\gamma_1 =  \frac{1}{(\lambda + 1) n}$. The function $h(X) = \lambda \|X\|_1 + \delta_{{\rm Ob}(n,p)}(X)$ is strongly prox-regular with respect to $\mathcal{C}_1$, $\| \cdot \|_{2,\infty}$ and $\gamma_1$. Moreover, the proximal mapping ${\rm prox}_{th}$ is semismooth over the set $\mathcal{D}_1 = \{X + tV: X \in \mathcal{C}_1, \|V\|_{2,\infty} = 1, 0 \leq t \leq \gamma_1 \}$.
    \item[{\rm (ii)}] Let $\mathcal{C}_2 = {\rm Ob}^+(n,p)$ and $0 < \gamma_2 < 1$. The function $h(X) = \delta_{{\rm Ob}^+(n,p)}(X)$ is strongly prox-regular with respect to $\mathcal{C}_2$, $\| \cdot \|$ and $\gamma_2$. Moreover, the proximal mapping ${\rm prox}_{th}$ is semismooth over the set $\mathcal{D}_2 = \{X + tV: X \in \mathcal{C}_2, \|V\|_{2,\infty} = 1, 0 \leq t \leq \gamma_2 \}$.
\end{itemize}
\end{lemma}
\begin{proof}
Let us prove (i) and (ii), respectively.
\begin{itemize}
    \item[(i)] Note that for any vector $x \in \R^n$ with $\|x\|_2 = 1$, $\|x\|_\infty \geq 1 / \sqrt{n}$.
Following from the definition of the proximal mapping \eqref{eq:l1oblique}, we have for $t \leq \gamma_1$, the proximal mapping ${\rm prox}_{th}$ is single-valued and Lipschitz continuous over $\mathcal{D}_1$. In addition, its directional derivative at $X$ along $U \in \R^{n\times p}$
\[ (D(X)[U])_i = \frac{U_i^-}{\|\tilde{w}_i^-\|_2} -  \frac{\tilde{w}_i^- (\tilde{w}_i^-)^\top U_i^-}{\|\tilde{w}_i^-\|_2^3}, \]
where $\tilde{w}_i^- = w_i^{-}\odot \mathrm{sign}(X_i), \; U_i^- =  U_i \odot 1_{w_i \leq 0}$, and $1_{a \leq 0}$ returns $1$ if $a\leq 0$ and $0$ otherwise. Due to the Cauchy inequality, it holds that $$\iprod{(D(X)[U])_i}{U_i} = \left( \|U_i^-\|_2^2\|\tilde{w}_i^-\|_2^2 - ((\tilde{w}_i^-)^\top U_i^-)^2 \right)/\|\tilde{w}_i^-\|_2^3 \geq 0.$$
Then the elements of generalized Jacobian of $\mathrm{prox}_{h}$ at $X$ are positive semidefinite. 
The monotonicity of ${\rm prox}_{th}$ follows from the positive semidefiniteness \cite[Proposition 2.1]{schaible1996generalized}. 

Since the proximal mapping \eqref{eq:l1oblique} is separable with respect to the columns in $X$, its semismoothness property can be reduced to the case of $p = 1$. Note that the nondifferential points of ${\rm prox}_{th}$ are in the set $\mathcal{A} := \{ x \in \R^n: \exists i,\; |x_i| = \lambda t \}$. At a nondifferentiable point $x \in \mathcal{A}$, let $d \in \R^n$ be a direction. Without loss of generality, assume $x_i = t\lambda$ and $|x_j| \ne t \lambda$ for all $j \ne i$. If $d_i > 0$, we have $\partial {\rm prox}_{th}(x+d) = \frac{\diag(1_{\tilde{w}< 0})}{\| \tilde{w}^-\|_2} - \frac{\tilde{w}^{-} (\tilde{w}^-)^\top}{\|\tilde{w}^-\|_2^3}=:J(x+d)$, with $\tilde{w}^- = \min(\lambda t - |x + d|, 0) \odot {\rm sign}(x)$. Define $\tilde{d}_j = d_j$ if $j \ne i$ and 0 otherwise. Note that ${\rm prox}_{th}(x+d) = {\rm prox}_{th}(x+\tilde{d})$, $J(x+d) = J(x + \tilde{d})$ and $J(x+d)d = J(x+\tilde{d})\tilde{d}$. Thus,
\[ \begin{aligned}
   & \|{\rm prox}_{th}(x+d) - {\rm prox}_{th}(x) - J(x+d) d\| \\
   = & \|{\rm prox}_{th}(x+\tilde{d}) - {\rm prox}_{th}(x) - J(x+\tilde{d}) \tilde{d}\| \\
   = & o(\|\tilde{d}\|) = o(\|d\|).
\end{aligned}
 \]
 One can draw a similar conclusion for the case $d_i < 0$. Combining them together, we conclude that ${\rm prox}_{th}$ is semismooth.
 \item[(ii)] It follows from the definition of the proximal mapping \eqref{eq:negoblique} that $\mathrm{prox}_{th}(X)$ is single-valued and Lipschitz continuous over $\mathcal{D}_2$. Furthermore, its directional derivative of $\mathrm{prox}_{th}(X)$ at $X$ along $U\in \R^{n \times p}$ is
\[ [D(X)[U]]_i = \frac{U_i^+}{\|X_i^+\|_2} - \frac{X_i^+(X_i^+)^\top U_i^+}{\|X_i^+\|_2^3},\]
where $U_i^+ =  U_i \odot 1_{X_i>0} + U_i \odot s$, and $s_i$ returns $1$ if $X_{ij} = 0, \; U_i >0$ and $0$ otherwise. Due to the Cauchy inequality, it holds that $$\iprod{(D(X)[U])_i}{U_i} = \left( \|U_i^+\|_2^2\|X_i^+\|^2 - ((X_i^+)^\top U_i^+)^2 \right)/\|X_i^+\|_2^3 \geq 0.$$
Then the elements of generalized Jacobian of $\mathrm{prox}_{th}$ at $X$ are positive semidefinite. Consequently, it is monotone. Analogous to the case above, one can prove the semismooth property of ${\rm prox}_{th}$.
\end{itemize}
\end{proof}
The strong prox-regularity and semismoothness established in the above lemma allow us to design efficient second-order methods for solving the applications in Subsection \ref{sec:motivate-exam}. Corresponding numerical experiments will be conducted in Section \ref{sec:num}.

\section{A projected Semismooth Newton method}\label{sec:ssn}
To solve \eqref{prob}, the proximal gradient method is
\be \label{eq:pgd} x^{k+1} \in \argmin_{x} \; \iprod{\nabla f(x^k)}{x -x^k} + \frac{1}{2t_k}\|x - x^k\|_2^2 + h(x) = {\rm prox}_{t_k h}(x^k - t_k \nabla f(x^k)),  \ee
where $t_k > 0$ is the step size depending on the Lipschitz constant of $\nabla f$. Since $h$ is nonconvex, $\mathrm{prox}_{t_kh}$ is usually a set-valued mapping. To accelerate \eqref{eq:pgd}, the author \cite{yang2017proximal} investigates the techniques of extrapolation and nonmonotone line search.

If $h$ is strongly prox-regular, the operator  $\mathrm{prox}_{th}(x^k - t\nabla f(x^k))$ is single-valued, Lipschitz continuous, and monotone (SLM) whenever $\|t \nabla f(x^k) \| \leq \gamma$ and $x^k$ belongs to the compact set $\mathcal{C}$. To ensure the compactness of the sequence $\{ x^k\}$, one usually investigates the coercive property and the descent property of $\varphi$. Specifically, any level set $\{ x : \varphi(x) \leq \alpha \}$ with $\alpha \in \R$ is compact for a coercive $\varphi$. If the sequence $\{\varphi(x^k)\}$ is decreasing, $\{x^k\} \subset \{x: \varphi(x) \leq \varphi(x^0)\}$ is a compact set. Moreover, the norm $\|\nabla f(x)\|$ is upper bounded over $\{x: \varphi(x) \leq \varphi(x^0) \}$ due to the smoothness. Denote by $L$ the upper bound of $\|\nabla f(x^0)\|$. The proximal operator $\mathrm{prox}_{th}$ is SLM if $t \leq \frac{\gamma}{L}$. For this choice of $t$, we are able to design a second-order method to solve the fixed point equation:
\be \label{eq:residual}  0 = F(x) := x - \mathrm{prox}_{th}(x - t \nabla f(x)), \ee
where $t$ is set as $\min\{\gamma,1\}/L$. It follows the SLM property of $\mathrm{prox}_{th}$ and twice continuous differentiability of $f$ that $F$ is single-valued, Lipschitz continuous, and semismooth.

Since $F$ is semismooth, we propose a semismooth Newton method for solving \eqref{prob}. One typical benefit of second-order methods is the superlinear or faster local convergence rate.  
Specifically, we first solve the linear system
\be \label{eq:ssn} (M^k + \mu_k I) d^k = - F(x^k), \ee
where $M^k \in M(x^k)$ defined by \eqref{eq:jacob-f} is a generalized Jacobian and $\mu_k = \kappa \|F(x^k)\|$ with a positive constant $\kappa$. The semismooth Newton step is then defined as
\be \label{eq:ssn-step} z^{k} = \mathcal{P}_{\mathrm{dom}(h)}(x^k + d^k), \ee
where the projection onto $\mathrm{dom}(h)$ is necessary for the globalization due to the nonconvexity of $h$. We remark that the strong prox-regularity in Definition \ref{assum:h} is crucial for the design of semismooth Newton methods. For a general prox-regular function $h$, we know from Proposition \ref{prop:Lip-prox-map-weakly-convex} that for $v \in \partial h(x)$, the proximal operator $\mathrm{prox}_{t h}$ is a singleton and Lipschitz continuous around $x+tv$ for sufficiently small $t$. Since $\nabla f(x)$ could be far away from $\partial h(x)$, the proximal operator ${\rm prox}_{t h}(x- t \nabla f(x))$ may not be a singleton. On the other hand, a uniform $t$ for all iterates may not exist. This non-singleton property causes difficulty in designing second-order methods.

Note that the pure semismooth Newton step is generally not guaranteed to converge from arbitrary starting points. For globalization, we switch to the proximal gradient step when the semismooth Newton step does not decrease the norm of the residual \eqref{eq:residual} or increases the objective function value to a certain amount. To be specific, the Newton step $z^k$ is accepted if the following conditions are simultaneously satisfied:
\begin{eqnarray}
\|F(z^k)\| & \leq & \nu \rho_k,  \label{eq:decrease-1} \\
\varphi(z^k) & \leq & \varphi(x^k) + \eta \rho_k^{1-q}\|F(z^k)\|^q, \label{eq:increase-1}
\end{eqnarray}
where $\rho_k$ is the normal of the residual of the last accepted Newton iterate until $k$ with an initialization $\rho_0 > 0$, $\eta > 0$, and $\nu, q \in (0,1)$. Otherwise, the semismooth Newton step $z^k$ fails, and we do a proximal gradient step, i.e.,
\be \label{eq:proxg-step} x^{k+1} = \mathrm{prox}_{t h} (x^k - t \nabla f(x^k)) = x^k -  F(x^k). \ee
Due to the choice of $t = \min\{ \gamma,1\}/L$, we will show in the next section that there is a sufficient decrease in the objective function value $\varphi(x^{k+1})$. Under the BD-regularity condition (Any element of $\partial_B F(x^*)$ at the stationary point $x^*$ is nonsingular \cite{qi1993convergence,pang1993nonsmooth}), we show in the next section that the semismooth Newton steps will always be accepted when the iterates are close to the optimal solution.
The detailed algorithm is presented in Algorithm \ref{alg:ssn}.

\begin{algorithm}[htbp]
\caption{A projected semismooth Newton method for solving \eqref{prob}} \label{alg:ssn}
\begin{algorithmic}[1]
\REQUIRE The constants $L>0$, $\gamma >0$, $\nu \in (0,1), q \in (0,1),  \eta > 0, \rho_0 > 0$, $\kappa >0$, and an initial point $x^0\in \mathbb{R}^{n}$, set $k = 0$.
\WHILE {the condition is not met}
\STATE Calculate the semismooth Newton direction $d^k$ by solving the linear equation
$$
(M(x^{k}) + \mu_k I) d^k = - F(x^k).
$$
\STATE Set $z^k = \mathcal{P}_{\mathrm{dom}(h)}(x^k + d^k)$. If the conditions \eqref{eq:decrease-1} and \eqref{eq:increase-1} are satisfied, set $x^{k+1} = z^k$. Otherwise, set $x^{k+1} = x^k -  F(x^k)$.
\STATE Set $k=k+1$.
\ENDWHILE
\end{algorithmic}
\end{algorithm}

\section{Convergence analysis} \label{sec:conver}
In this section, we will present the convergence properties of the proposed projected semismooth Newton method, i.e., Algorithm \ref{alg:ssn}. It consists of two parts,
the global convergence to a stationary point from any starting point and the local superlinear convergence.
\subsection{Global convergence}
First of all, we introduce the following assumptions.
\begin{Assumption} \label{assum:f}
    For problem \eqref{prob},  we assume
    \begin{itemize}
        \item the function $f$ is twice continuously differentiable, its gradient $\nabla f$ is Lipschitz continuous with modulus $L > 0$.
        \item the function $h$ is strongly prox-regular with respect to $\mathcal{C}$ and $\gamma$.
        \item the function $\varphi$ is bounded from below and coercive.
    \end{itemize}
\end{Assumption}

With the above assumption, the proximal gradient step \eqref{eq:proxg-step} leads to a sufficient decrease on $\varphi$.
\begin{Lemma} \label{lem:dec}
Suppose that Assumption \ref{assum:f} holds. Then for any $t_k \in (0, \frac{1}{L}]$
we have
\be \label{eq:dec-proxg} \varphi(x^k) - \varphi(x^{k+1}) \geq \left(\frac{1}{2t_k} - \frac{L}{2} \right) \|x^{k+1} - x^k\|^2. \ee
\end{Lemma}
\begin{proof}
It follows from the optimality of $x^{k+1}$ that
\[ \iprod{\nabla f(x^k)}{x^{k+1} - x^k} + \frac{1}{2t_k}\|x^{k+1} - x^k\|^2 + h(x^{k+1}) \leq h(x^k). \]
By Assumption \ref{assum:f} and $t_k \in (0, \frac{1}{L})$, we have
\[ \begin{aligned}
f(x^{k+1}) + h(x^{k+1}) & \leq f(x^k) +  \iprod{\nabla f(x^k)}{x^{k+1} - x^k} + \frac{L}{2}\|x^{k+1} - x^k\|^2 + h(x^{k+1}) \\
& \leq f(x^k) + h(x^k) + \left(\frac{L}{2} - \frac{1}{2t_k} \right)  \|x^{k+1} - x^k\|^2.
\end{aligned}\]
The proof is completed.
\end{proof}

From the above lemma, the convergence of the proximal gradient method for solving \eqref{prob} can be obtained by the coercive property of $\varphi$. When the projected semismooth Newton update $z^k$ is accepted, the function value $\varphi(z^k)$ may increase while the residual decreases as guaranteed by \eqref{eq:decrease-1} and \eqref{eq:increase-1}. This allows us to show global convergence.
\begin{theorem} \label{thm:global-con}
Let $\{x^{k}\}$ be the iterates generated by Algorithm \ref{alg:ssn}. Suppose that Assumption \ref{assum:f} holds. Let $t_k \equiv  t \in \left(0, \min(\gamma, 1)/L \right]$,
Then we have
\[ \lim_{k \rightarrow \infty} \|F(x^k)\| = 0. \]
\end{theorem}
\begin{proof}
If $x^{k+1}$ is obtained by the proximal gradient update, it holds from Lemma \ref{lem:dec} that
\be \label{eq:dec1} \varphi(x^k) - \varphi(x^{k+1}) \geq \left( \frac{1}{2t} - \frac{L}{2} \right) \|F(x^k)\|^2. \ee
It follows the Lipschitz properties of $\mathrm{prox}_{th}$ and $\nabla f(x)$ that $F$ is Lipschitz continuous. Let $L_F$ be the Lipschitz constant of $F$. From the triangle inequality, we have
\[ \| F(x^{k+1}) \| \leq \|F(x^k)\| + \|F(x^{k+1}) - F(x^k)\| \leq (L_F + 1)\|F(x^k)\|. \]
Plugging the above inequality into \eqref{eq:dec1} leads to
\be \label{eq:dec2}
\varphi(x^k) - \varphi(x^{k+1}) \geq c_1 \|F(x^{k+1})\|^2,
\ee
where $c_1: = \left( \frac{1}{2t} - \frac{L}{2} \right)\frac{1}{(L_F + 1)^2} > 0$.

If the Newton update $z^k$ is accepted, the conditions \eqref{eq:decrease-1} and \eqref{eq:increase-1} imply that
\[ \begin{aligned}
   \varphi(x^{k}) - \varphi(x^{k+1}) & \geq - \eta \rho_k^{1-q} \|F(x^{k+1})\|^q \\
   & = c_1 \|F(x^{k+1})\|^2 - (c_1\|F(x^{k+1})\|^{2-q} + \eta \rho_{k}^{1-q}) \|F(x^{k+1})\|^q
\end{aligned}  \]
and $\rho_{k+1} = \| F(x^{k+1}) \| \leq \nu \rho_k$. Since $\rho_{k} \in (0, \rho_0)$ for all $k$, $c_1\|F(x^{k+1})\|^{2-q} + \eta \rho_{k}^{1-q}$ is bounded by a constant, denoted by $c_2$. Hence, for the projected semismooth Newton step, it holds
\be \label{eq:dec3} \varphi(x^{k}) - \varphi(x^{k+1}) \geq c_1 \|F(x^{k+1})\|^2 - c_2 \rho_{k+1}^q.
\ee
Combining \eqref{eq:dec2} and \eqref{eq:dec3}, we have
\[ \varphi(x^0) - \varphi(x^{K+1}) = \sum_{i=1}^K(\varphi(x^k) - \varphi(x^{k+1})) \geq c_1\sum_{k=0}^K \|F(x^{k+1})\|^2 - c_2 \sum_{k \in \mathcal{K}_N} \rho_{k+1}^q, \]
where $\mathcal{K}_N \subset \{ 1,2,\ldots, K+1 \}$ consists of the indices where the projected semismooth Newton updates are accepted. It is easy to see that $\sum_{k \in \mathcal{K}_N} \rho_{k+1}^q \leq \rho_0^q \sum_{k=1}^{K+1} \nu^{qk} = \frac{\rho_0^q(1-\nu^{q(K+1)})}{1-\nu} \leq \frac{\rho_0^q }{1-\nu^q}$. Therefore,
\[ c_1 \sum_{k=0}^K \|F(x^{k+1})\|^2 \leq \varphi(x^0) - \varphi(x^{K+1}) + \frac{c_2\rho_0^p}{1-\nu^p}.  \]
Since $\varphi$ is bounded from below, we have
\[ \sum_{k=0}^\infty \|F(x^k)\|^2 \leq \infty, \]
which implies that $\lim_{k \rightarrow \infty}\| F(x^k)\| = 0$. We complete the proof.
\end{proof}

\subsection{Local superlinear convergence}
The local superlinear convergence of the semismooth Newton update has been studied in \cite{qi1993nonsmooth,qi1999survey,xiao2018regularized}. The difficulties in our case lie in the extra nonconvex projection operator $\mathcal{P}_{\mathrm{dom}(h)}$ and the switching conditions \eqref{eq:decrease-1} and \eqref{eq:increase-1}. We make the following assumptions.
\begin{Assumption} \label{assum:local}
Let $\{x^k\}$ be the iterates generated by Algorithm \ref{alg:ssn}.
\begin{itemize}
    \item[{\rm (A1)}] The iterate $x^k$ converges to $x^*$ with $F(x^*) = 0$, as $k \rightarrow \infty$.
    \item[{\rm (A2)}] The Hessian $\nabla^2 f$ is continuous around $x^*$.
    \item[{\rm (A3)}] The mapping $F$ is semismooth at $x^*$. In addition, there exists $C> 0$ such that each element $M \in M(x^*)$ defined by \eqref{eq:jacob-f} is nonsingular with $\|M^{-1}\| \leq C$.
    \item[{\rm (A4)}] 
    The function $\varphi$ is Lipschitz continuous over $\mathrm{dom}(h)$ with modulus $L_{\varphi}$, i.e., for all $x,y \in \mathrm{dom}(h)$,
    \[ |\varphi(x) - \varphi(y)| \leq L_{\varphi}\|x-y\|. \]
\end{itemize}
\end{Assumption}

Since the convergence of $\{\|F(x^k)\|\}$ is proved in Theorem \ref{thm:global-con}, any accumulation point of $\{x^k\}$ has zero residual. The Assumption (A1) reads that the full sequence $\{x^k\}$ is convergent. The Assumption (A2) holds for any twice continuously differentiable $f$. The Assumption (A3) is the standard BD-regularity condition used in \cite{qi1993convergence,pang1993nonsmooth,milzarek2014semismooth,xiao2018regularized}.

For the projection operator $\mathcal{P}_{\mathrm{dom}(h)}$ in Algorithm \ref{alg:ssn}, we prove the following bounded property, \revise{which has also been used in the convergence rate analysis for the generalized power method for the group synchronization problems \cite[Lemma 1]{liu2017discrete} \cite[Proposition 3.3]{liu2017estimation} \cite[Lemma 2]{liu2020unified}.}
\begin{Proposition}
For all $x \in \R^n$ and $y \in \mathrm{dom}(h)$, it holds $\|\mathcal{P}_{\mathrm{dom}(h)} (x) - y\| \leq 2\|x - y\|$.
\end{Proposition}
\begin{proof}
Following the definition of $\mathcal{P}_{\mathrm{dom}(h)}$, we have
\[ \|\mathcal{P}_{\mathrm{dom}(h)}(x) - y\| \leq \| \mathcal{P}_{\mathrm{dom}(h)}(x) - x  \| + \|x - y\| \leq 2\|x - y\|.  \]
\end{proof}

The following lemma shows that the switching conditions \eqref{eq:decrease-1} and \eqref{eq:increase-1} are satisfied by the projected semismooth Newton update when $k$ is large enough.
\begin{lemma} \label{lem:Newton-acc}
Let $\{x^k\}$ be the iterates generated by Algorithm \ref{alg:ssn}. Suppose that Assumptions \ref{assum:f} and \ref{assum:local} hold. Then for sufficiently large $k$, the Newton update $z^k$ is always accepted.
\end{lemma}
\begin{proof}
Let us first define a constant $\gamma_F \in \left(0, \min\left\{ \frac{1}{8C}, \frac{\nu}{32C^2L_F}, \frac{\eta^{\frac{1}{1-q}}}{32C^2(L_{\varphi} 3^q C^q)^{\frac{1}{1-q}}} \right\}\right)$, where $C, \nu, \eta, q, L_F, L_\varphi$ are defined previously. It follows from \cite[Lemma 2.6]{qi1993convergence} and (A3) that
% Due to the semismoothness of $F$ at $x^*$ and the outer continuity of $\partial F(x)$, 
there exists $\varepsilon > 0$ such that for any $x \in \mathbb{B}(x^*,\varepsilon)$ and $M \in M(x)$, 
\be \label{eq:semismooth-local} \|F(x) - F(x^*) - (M + \kappa \|F(x)\| I) (x-x^*)\| \leq \gamma_F \|x - x^*\|,  \; \|(M + \kappa \|F(x)\|I)^{-1}\| \leq 2C. \ee
For the projected semismooth Newton update $z^k = \mathcal{P}_{\mathrm{dom}(h)} (x^k - (M_k + \mu_k I)^{-1} F(x^k))$, it hold that
\be \label{eq:zr-recur} \begin{aligned}
   \|z^k - x^*\| & = \| \mathcal{P}_{\mathrm{dom}(h)} (x^k - (M_k + \mu_k I)^{-1} F(x^k)) - x^*  \|  \\
   & \leq 2 \| (M_k + \mu_k I )^{-1} (F(x^k) - F(x^*) - (M_k + \mu_k I)(x^k - x^*)) \| \\
   & \leq 4 \gamma_F C \|x^k - x^*\|,
\end{aligned}
\ee
where we assume $x^k \in \mathbb{B}(x^*,\varepsilon)$. Due to the choice of $\gamma_F$, we have $z^k \in \mathbb{B}(x^*,\varepsilon)$. Note that
\be \|x^k - x^*\|  \leq \|z^k - x^* \| + \|z^k - x^k\| \leq  4  \gamma_F C \|x^k - x^*\| + 4C \|F(x^k)\|. \ee
Then
\be \label{eq:xdist}\|x^k -x^*\| \leq \frac{4C}{1-4 \gamma_F C} \|F(x^k)\|. \ee
Combining \eqref{eq:zr-recur} and \eqref{eq:xdist} implies
\be \label{eq:zdist} \|z^k - x^*\| \leq \frac{16 \gamma_F C^2 }{1-4\gamma_F C} \|F(x^k)\|. \ee
Hence,
\be \label{eq:residual1}
   \|F(z^k)\| = \|F(z^k) - F(x^*)\| \leq L_F \|z^k - x^*\| \leq \frac{16 \gamma_F C^2 L_F}{1-4\gamma_F C} \|F(x^k)\| \leq \nu \|F(x^k)\|.
\ee
In addition, note that
\[ \begin{aligned}
\|z^k - x^*\| & = \|(M_k + \mu_k I)^{-1}\left(F(z^k) - F(x^*) - (M_k + \mu_k I)(z^k - x^*) - F(z^k)\right)\| \\
& \leq 2 
\gamma_FC\|z^k - x^*\| + 2C \|F(z^k)\|.
\end{aligned}
\]
This gives
\be \label{eq:zdist1} \|z^k - x^*\| \leq \frac{2C}{1-2\gamma_F C}\|F(z^k)\|.\ee
The changes between $\varphi(z^k)$ and $\varphi(x^k)$ can be estimated by
\be \label{eq:psidist}
\begin{aligned}
   \varphi(z^k) - \varphi(x^k) & \leq  \varphi(z^k) - \varphi(x^*)  \leq L_{\varphi} \|z^k -x^*\| \\
   & = L_{\varphi} \|z^k -x^*\|^{1-q} \|z^k-x^* \|^q \\
   & \leq L_{\varphi}  \left(\frac{16 \gamma_F C^2}{1-4\gamma_F C} \right)^{1-q} \left(\frac{2C}{1-2\gamma_F C}\right)^q \|F(x^k)\|^{1-q} \|F(z^k)\|^q \\
   & \leq \eta \|F(x^k)\|^{1-q} \|F(z^k)\|^q.
\end{aligned}\ee
Due to the convergence of residual, for any proximal gradient step index $k_0$, there always exists a $k > k_0$ such that $\|F(x^k)\| \leq \rho_k$. Then all followed iterates are projected semismooth Newton steps because of \eqref{eq:residual1} and \eqref{eq:psidist}. This completes the proof.
\end{proof}

The above lemma establishes the local transition to the projected semismooth Newton step. Utilizing the semismoothness, we have the locally superlinear convergence on the iterates generated by Algorithm \ref{alg:ssn}.
\begin{theorem}
Let $\{x^k\}$ be the iterates generated by Algorithm \ref{alg:ssn}. Suppose that Assumptions \ref{assum:f} and \ref{assum:local} hold. Then there exists a finite $K > 0$, such that for all $k \geq K$, $\{x^k\}$ converges to $x^*$ Q-superlinearly.
\end{theorem}
\begin{proof}
From Lemma \ref{lem:Newton-acc}, there exists a $K$ such that the projected semismooth Newton update is accepted for $k \ge K$. It follows from the semismoothness of $F$ that
\[ \begin{aligned}
\|x^{k+1} - x^k\| & = \|\mathcal{P}_{\mathrm{dom}(h)} (x^k - (M_k + \mu_k I)^{-1} F(x^k)) - x^* \|  \\
& \leq 4C \| F(x^k) - F(x^*) - (M_k+ \mu_k I )(x^k -  x^*) \| \\
& = o(\|x^k - x^*\|),
\end{aligned}
\]
which means $\{x^k\}$ converges to $x^*$ Q-superlinearly.
\end{proof}

\section{Numerical experiments}\label{sec:num}
In this section, some numerical experiments are presented to evaluate the performance of our proposed Algorithm \ref{alg:ssn}, denoted by ProxSSN. We compare ProxSSN with the existing methods including  AManPG and ARPG \cite{huang2021riemannian}. We also test the proximal gradient descent method (ProxGD for short) as in \eqref{eq:pgd}. Here, a nonmonotone line search with Barzilai–Borwein (BB) step size \cite{barzilai1988two} is used for acceleration. Let $s^k = x^{k} - x^{k-1}$ and $y^k = \nabla f(x^k) - \nabla f(x^{k-1})$. The BB step sizes are defined as
\be \label{eq:bb}
\beta_k^1 = \frac{\iprod{s^k}{s^k}}{|\iprod{s^k}{y^k}|}, \; \text{and} \; \beta_k^2 = \frac{|\iprod{s^k}{y^k}|}{\iprod{y^k}{y^k}}.
\ee
Given $\varrho, \delta \in (0,1)$, the nonmonomote Armijo line search is to find the smallest nonnegative integer $\ell$ satisfying
\be \label{eq:ls} \varphi(\mathrm{prox}_{t_k(\ell)h}(x^k - t_k(\ell) \nabla f(x^k) )) \leq C_k + \frac{\varrho}{2t_k(\ell) } \|\mathrm{prox}_{t_k(\ell) h}(x^k - t_k(\ell) \nabla f(x^k) ) - x_k \|^2. \ee
Here, $t_k(\ell):= \beta_k \delta^{\ell}$, $\beta_k$ is set to $\beta_k^1$ and $\beta_k^2$ alternatively, and the reference value $C_k$ is calculated via $C_k = (\varpi Q_{k-1}C_{k-1} + \varphi(x^k))/Q_k$ where $\varpi \in [0,1]$, $C_0 = \varphi(x^0), \; Q_{k} = \varpi Q_{k-1} + 1$ and $Q_0 = 1$. Once $\ell$ is obtained, we set $t_k = \beta_k \delta^\ell$ and the next iterate is then given by $x^{k+1} = \mathrm{prox}_{t_k h}(x^k - t_k \nabla f(x^k))$.

The reasons of not using ManPG \cite{chen2018proximal}, RPG \cite{huang2021riemannian} or the algorithms proposed in \cite{Lai2014A,Kovna2015} is that their performance can not measure up with AManPG or ARPG in tests of \cite{huang2021riemannian}. For ARPG and AManPG, we use the code provided by \cite{huang2021riemannian}\footnote{all codes are available at \url{https://www.math.fsu.edu/~whuang2/files/RPG_v0.2.zip} }. The codes were written in MATLAB and run on a standard PC with 3.00 GHz AMD R5 microprocessor and 16GB of memory.  The reported time is wall-clock time in seconds.

\subsection{Sparse principal component analysis}\label{sec:spca}
In this subsection, we consider the sparse PCA problem \eqref{prob:spca}, which can be regarded as  a nonsmooth problem on the oblique manifold. Let $f(X): = \|X^TA^TAX - D^2 \|_F^2$. AManPG solves the following subproblem in each iteration:
\begin{equation*}
    \eta_{X^k} = \argmin_{\eta\in T_{X^k}\mathrm{Ob}(n,p)} \left< \mathrm{grad}f(X^k),\eta  \right> + \frac{\tilde{L}}{2}\|\eta\|_F^2 + \lambda \|X^k + \eta\|_1,
\end{equation*}
where $\tilde{L}>L$ with $L$ being the Lipschitz constant of $f$, $\mathrm{grad}f(X^k)$ denotes the Riemannian gradient of $f$ at $X^k$, and $T_{X}\mathrm{Ob}(n,p)$ is the tangent space to $\mathrm{Ob}(n,p)$ at $X$. We refer to \cite{chen2018proximal} for more details. In the $k$-th iteration of ARPG, one need to solve the subproblem:
\begin{equation*}
    \eta_{X^k} = \argmin_{\eta\in T_{X^k}\mathrm{Ob}(n,r)} \left< \mathrm{grad}f(X^k),\eta  \right> + \frac{\tilde{L}}{2}\|\eta\|_F^2 + \lambda\|\mathcal{R}_{X^k}(\eta)\|_1,
\end{equation*}
where $\mathcal{R}$ denotes a retraction operator on $\mathrm{Ob}(n,p)$. The termination condition of both AManPG and ARPG is as follows:
\begin{equation}\label{kkt3}
    \|\tilde{L} \eta_{X^k}\|^2 \leq \mathrm{tol},
\end{equation}
where $\mathrm{tol} > 0$ is a given tolorance. The ProxGD and ProxSSN methods are applied to solve problem \eqref{prob:spca} by setting $f(X): = \|X^TA^TAX - D^2 \|_F^2, \; h(X) = \lambda \|X\|_1 + \delta_{{\rm{Ob}(n,p)}}(X)$. ProxGD has the following update rule
\begin{equation*}
    X^{k+1}  = \mathrm{prox}_{t_kh}(X^k - t_k \nabla f(X^k)).
\end{equation*}
The following relative KKT condition  is set as a stopping criterion for our algorithm and ProxGD:
\begin{equation}\label{equ:stop kkt}
  \begin{aligned}
     {\rm err}:=\frac{\left\| X^k - \mathrm{prox}_{t_kh}(X^k -  t_k\nabla f(X^k)) \right\|_F}{t_k(1+\|X^k\|_F)} \leq \mathrm{tol}.
  \end{aligned}
\end{equation}
Note that $t_k$ is fixed in ProxSSN.
\paragraph{Implementation details}The parameters of AManPG and ARPG are set the same as in \cite{huang2021riemannian}. For ProxSSN, we set $q = 20,\nu = 0.9999,\eta = 10^{-6}, t = 1/\lambda_{\max}(A^TA)$, and the initial value $\kappa = 1$.  The maximum number of iterations is 10000. 
% We use the conjugate gradient (CG) method to solve the subproblem; the maximum number of iterations and the stopping condition are depended on the residual $F(X^k)$. The initial tolerance is set to 0.1 and the total maximum number of  CG  iterations is  20. 
The starting point of all algorithms is the leading $p$ right singular vectors of the matrix $A$.   Due to the evaluation criterion being different for different algorithms, we first run ARPG when \eqref{kkt3} is satisfied with $\mathrm{tol} = 10^{-10}nr$ or the number of iterations exceeds 10000, and denote $F_{\rm ARPG}$ as the obtained objective value. The other algorithms are terminated when the objective value satisfies $F(X^k)\leq F_{\rm ARPG} + 10^{-6}$ or \eqref{kkt3} (or \eqref{equ:stop kkt}) is satisfied with $\mathrm{tol} = 10^{-10}np$, or the number of iterations exceeds 10000.

In our experiments, the data matrix $A\in\mathbb{R}^{m\times n}$ is produced by MATLAB function $randn(m, n)$, in which all entries of $A$ follow the standard Gaussian distribution. Next, we shift the columns of $A$ such that they have zero-mean, and normalize the resulting matrix by its spectral norm.

\subsubsection{Numerical results} In Figure \ref{fig:perf_spca_time}, we present the trajectories of the objective function values with respect to the wall-clock time for the cases of $n=300$ and $n=400$, where $\varphi_{\min}$ is the minimum objective value of all algorithms in the iterative process. It can be seen that our proposed ProxSSN converges fastest among all algorithms. AManPG and ARPG have comparable performances. 
% This is consistent with the theoretical properties of second-order algorithms. 
Figures \ref{fig:perf_spca_n} and \ref{fig:perf_spca_r}  shows the performance of all algorithms under different $n,p$. We see that all algorithms have similar objective values, but the consuming time of ProxSSN is the least. We present the wall-clock time in the column ``time'' and the objective function value in the column ``obj'' in Table \ref{tab:spca} for different combinations of $m,n, p$, where similar conclusions can be drawn. 

We also compare the accuracy and efficiency of ProxSSN with other algorithms using the performance profiling method proposed in \cite{dolan2002benchmarking}.
Let $t_{i,s}$ be some performance quantity (e.g. the wall-clock time or the gap between the obtained objective function value and $\varphi_{\min}$, lower is better) associated with the $s$-th solver on problem $i$.
Then, one computes the ratio $r_{i,s}$ as $t_{i,s}$ over the smallest value obtained by $n_s$ solvers on problem $i$, i.e., $r_{i,s} :=\frac{t_{i,s}}{\min\{t_{i,s}: 1\leq s \leq n_s\}}$. For $\tau >0$, the value
\[
\pi_{s}(\tau) : = \frac{\text{number of problems where } \log_2(r_{i,s}) \leq \tau}{\text{total number of problems}}
\]
indicates that solver $s$ is within a factor $2^\tau \geq 1$ of the performance obtained by the best solver. Then the performance plot is a curve $\pi_s(\tau)$ for each solver $s$ as a function of $\tau$.  In Figure \ref{fig:profile_spca_time}, we show the performance profiles of the criterion, the wall-clock time and the gap in the objective function values. In particular, the intercept point of the axis ``ratio of problems'' and the curve in each subfigure is the percentage of the faster one among the four solvers. These figures show that both the wall-clock time and the gap in the objective function values of ProxSSN are much better than other algorithms on most problems.

\begin{figure}[!htb]
\centering
\subfigure{
\includegraphics[width=0.45\textwidth]{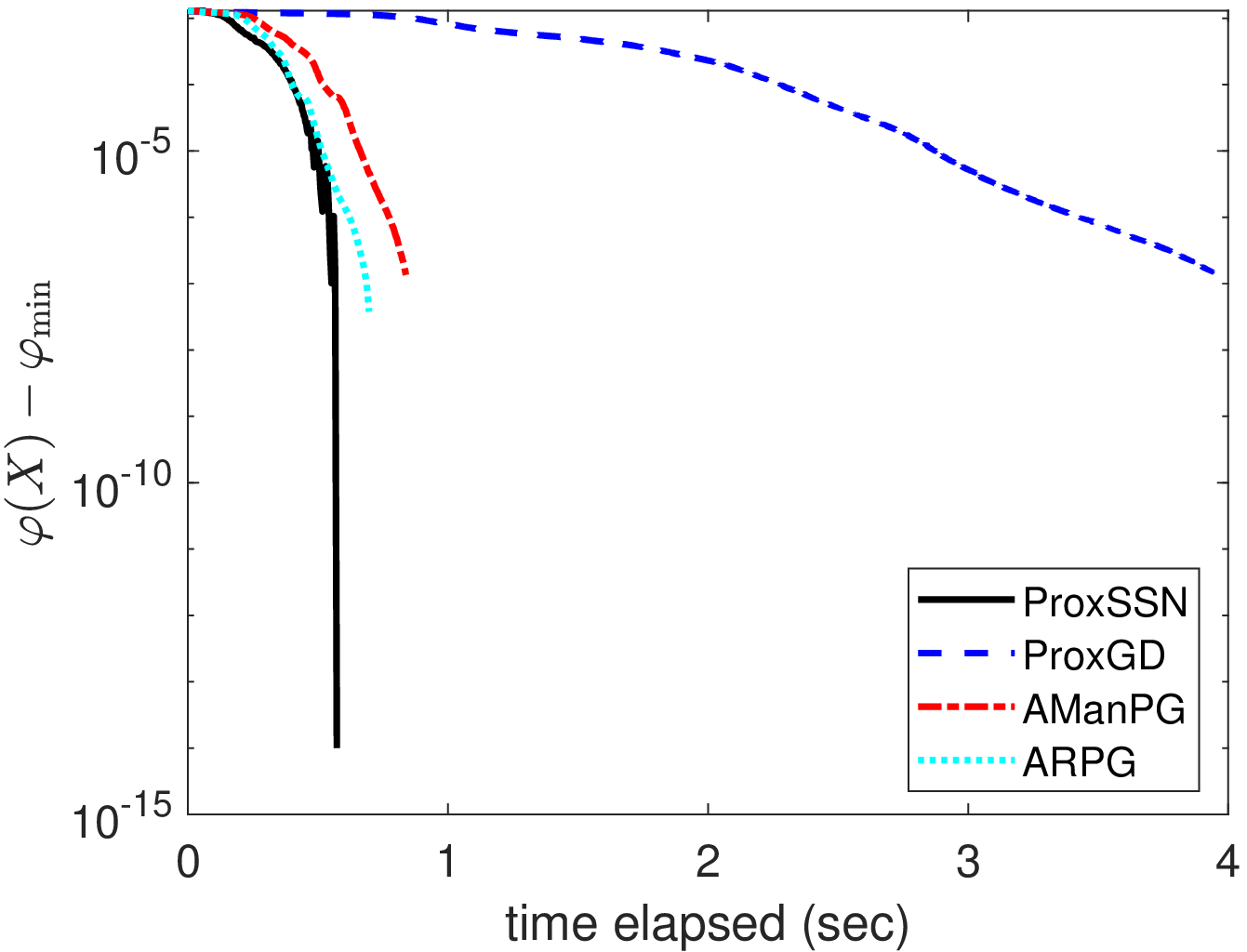}}
\subfigure{
\includegraphics[width=0.45\textwidth]{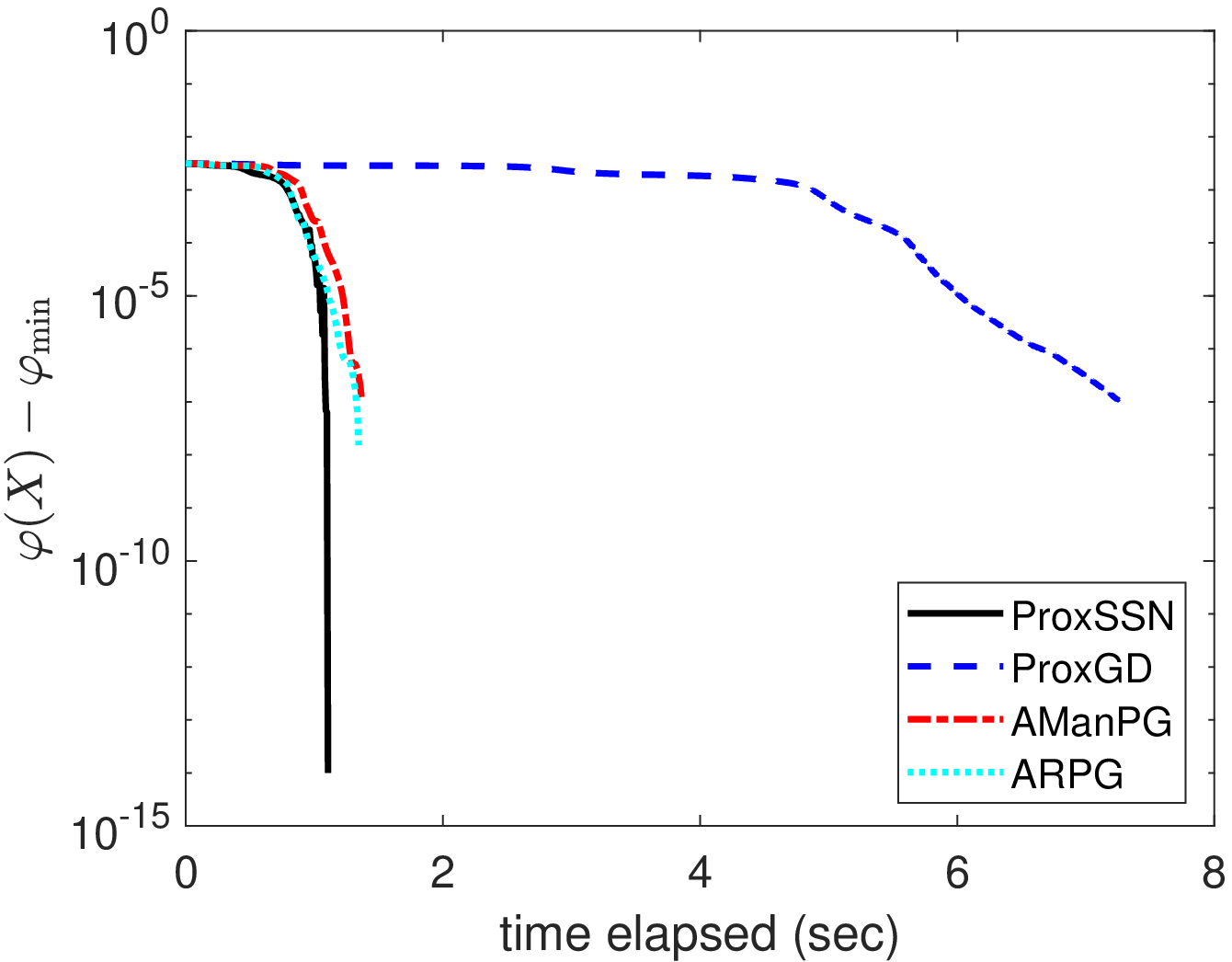}}
\caption{The trajectories of the objective function values with respect to the wall-clock time on the sparse PCA problem \eqref{prob:spca} with $p = 10,\lambda = 0.01$. Left: $n=300$; right: $n=400$}\label{fig:perf_spca_time}
\end{figure}

\begin{figure}[!htb]
\centering
\subfigure{
\includegraphics[width=0.45\textwidth]{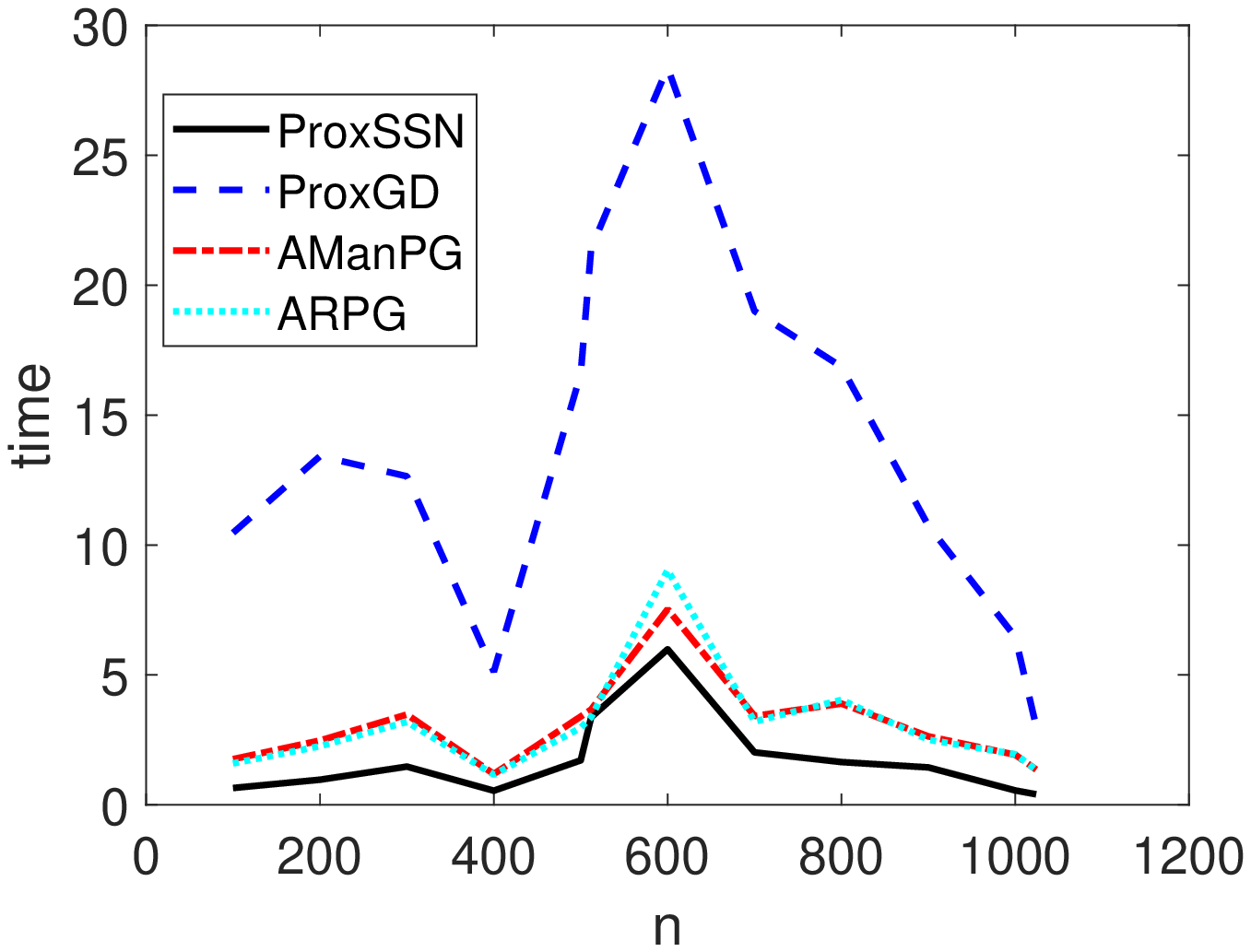}}
\subfigure{
\includegraphics[width=0.45\textwidth]{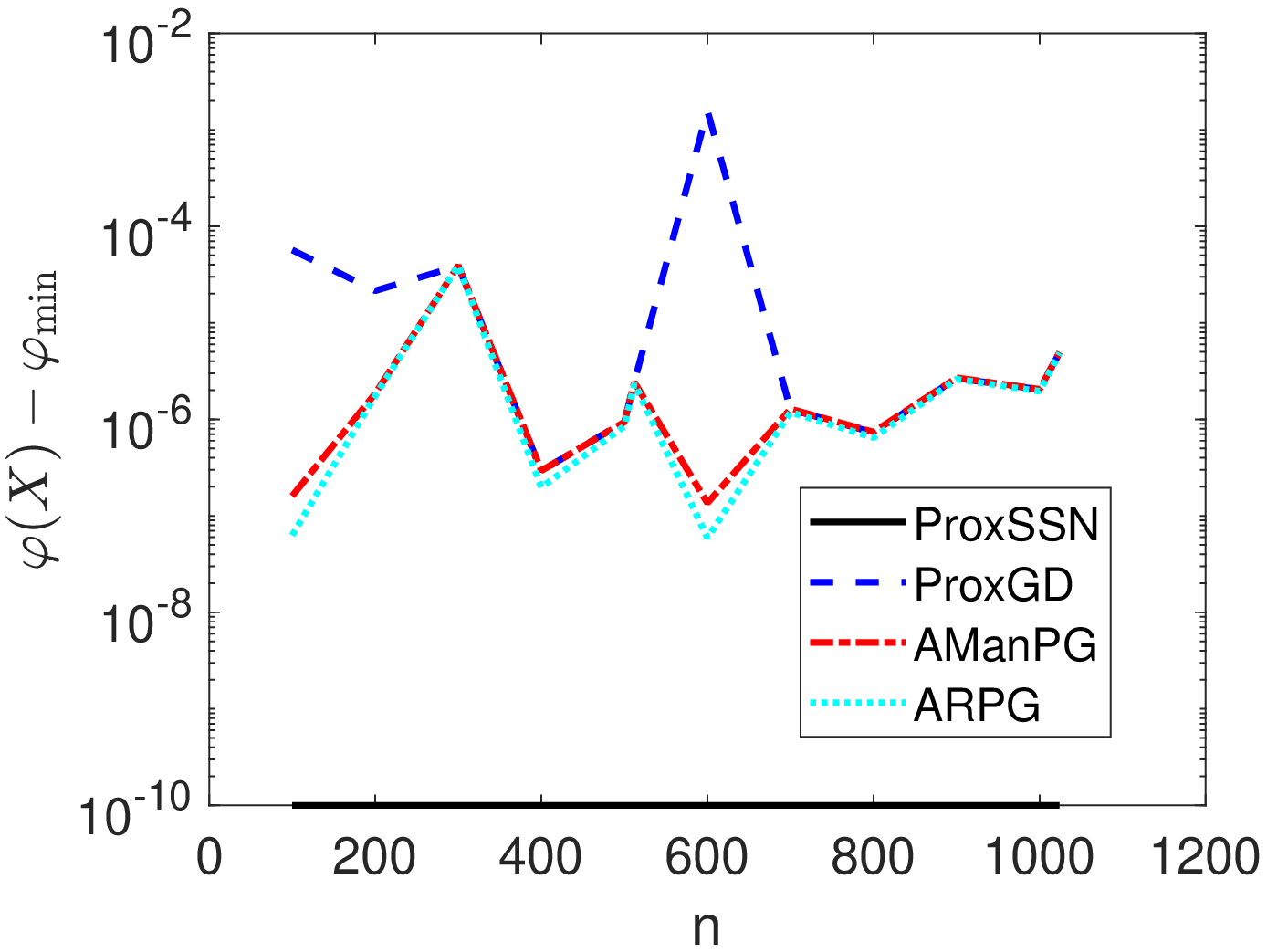}}
\caption{Comparisons of wall-clock time and the objective function values on the sparse PCA problem \eqref{prob:spca} with $p = 20,\lambda = 0.01$ for different $n$.}\label{fig:perf_spca_n}
\end{figure}

\begin{figure}[!htb]
\centering
\subfigure{
\includegraphics[width=0.45\textwidth]{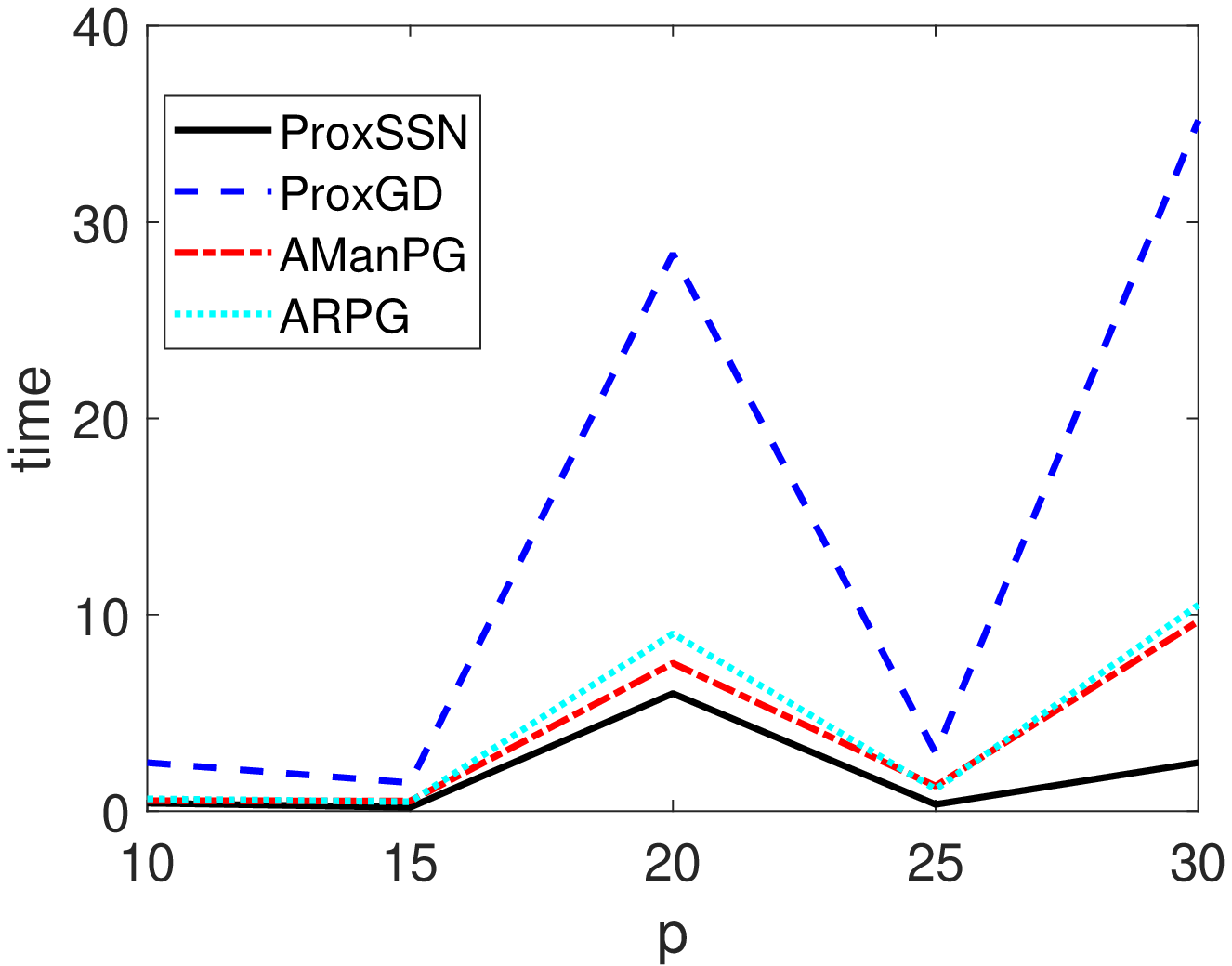}}
\subfigure{
\includegraphics[width=0.45\textwidth]{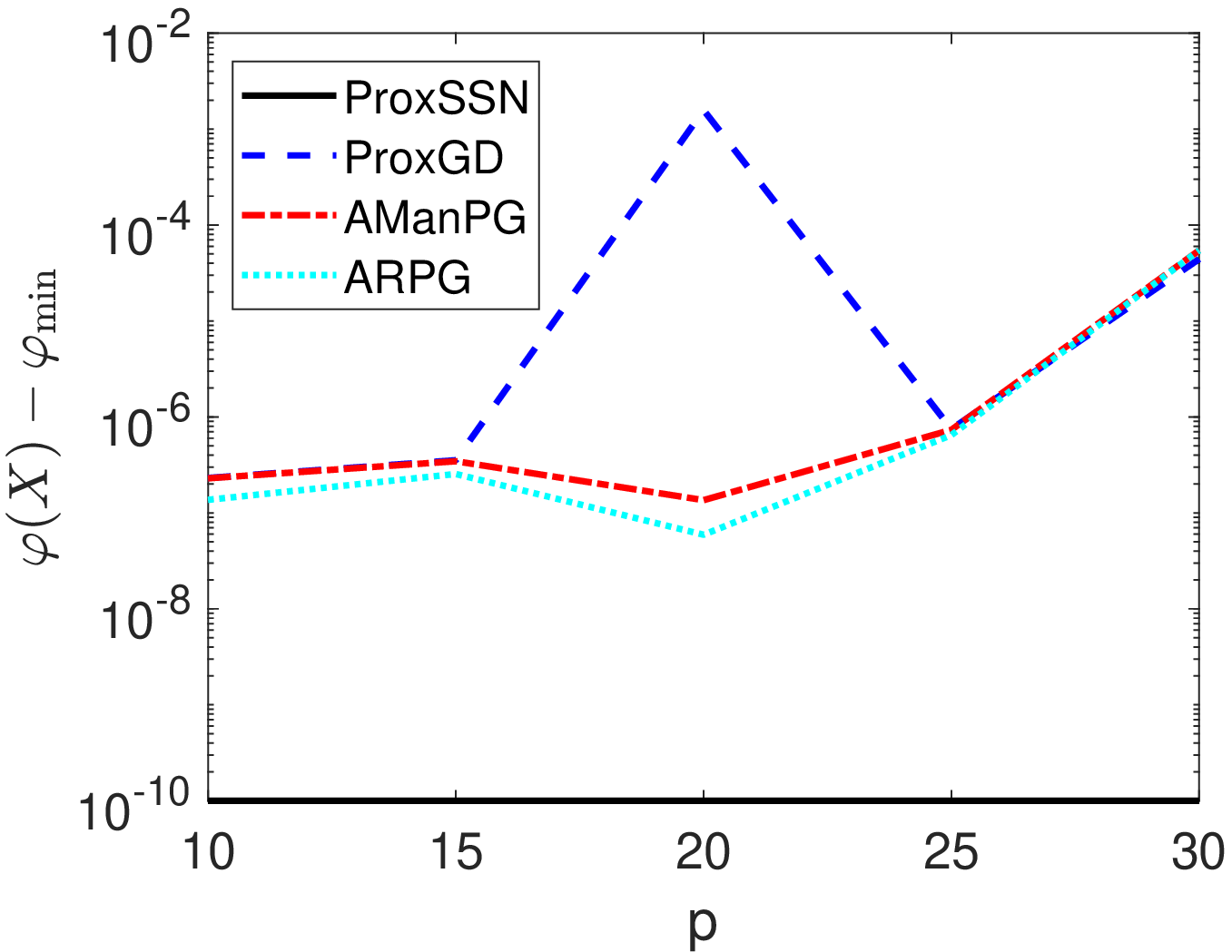}}
\caption{Comparisons of wall-clock time and the objective function values on the sparse PCA problem \eqref{prob:spca} with $n = 512,\lambda = 0.01$ for different $p$.}\label{fig:perf_spca_r}
\end{figure}

\begin{figure}[!htb]
\centering
\subfigure{
\includegraphics[width=0.45\textwidth]{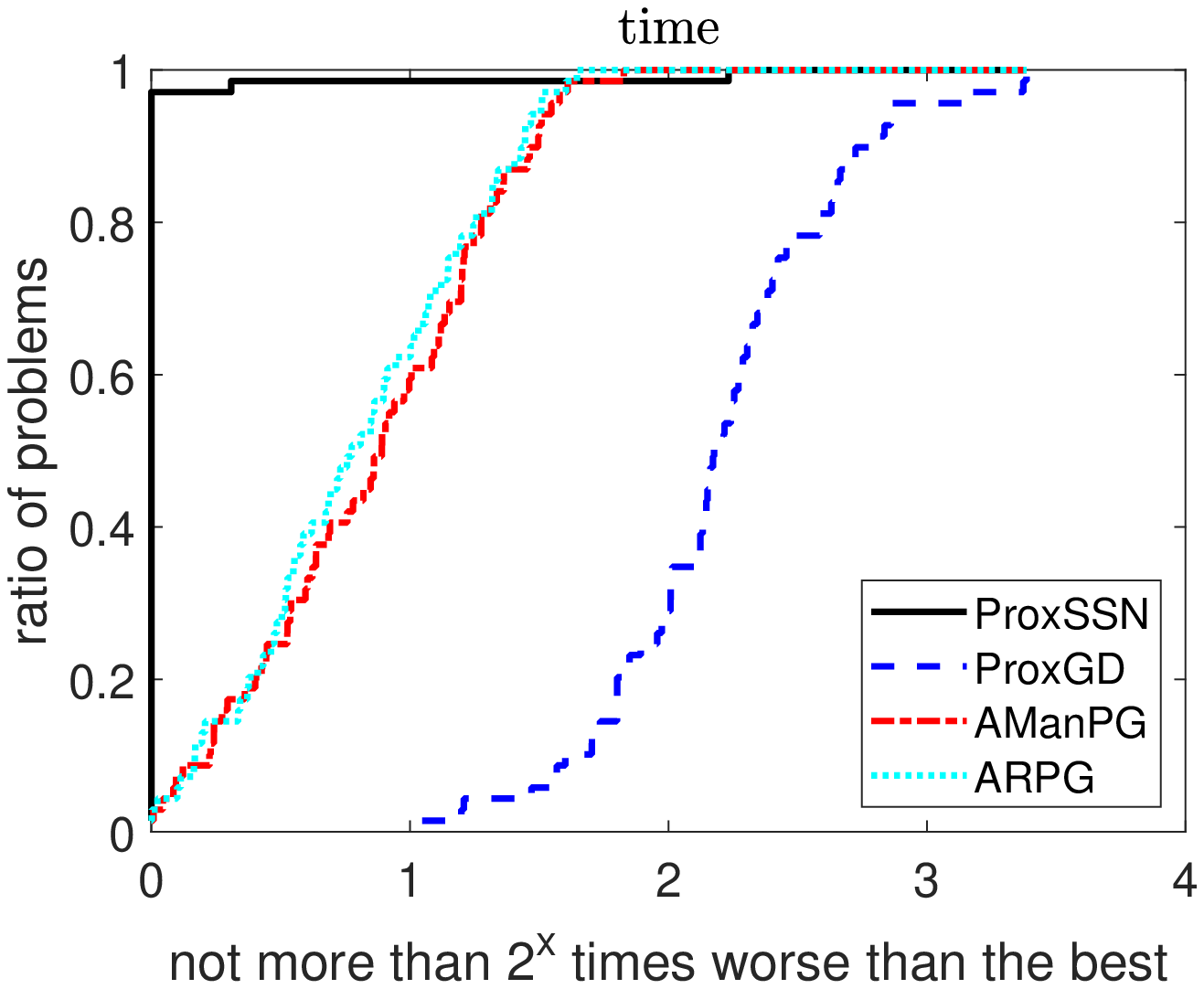}}
\subfigure{
\includegraphics[width=0.45\textwidth]{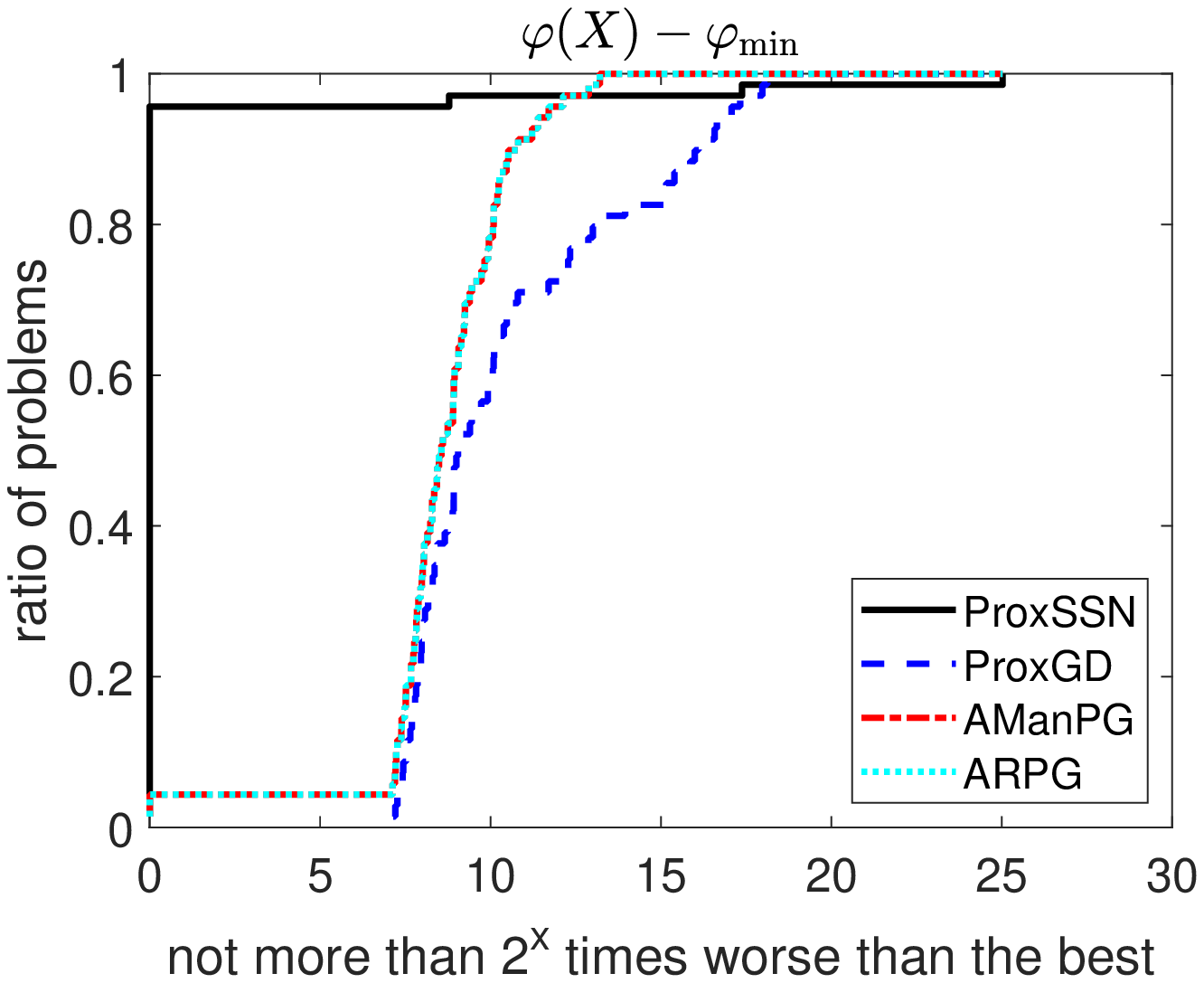}}
\caption{The performance profiles on the sparse PCA problem \eqref{prob:spca}.}\label{fig:profile_spca_time}
\end{figure}

\begin{footnotesize}
\setlength{\tabcolsep}{3pt}
%\footnotesize
\centering
\begin{longtable}{|c|cc|cc|cc|cc|}
\caption{Computational results of oblique SPCA}\label{tab:spca}\\ %\hline
\hline
\multirow{2}{*}{$(m,n,p)$} & \multicolumn{2}{c|}{ProxSSN} & \multicolumn{2}{c|}{ProxGD} & \multicolumn{2}{c|}{AManPG} & \multicolumn{2}{c|}{ARPG}    \\ \cline{2-9}
     & time   &  obj  & time & obj   & time & obj   & time & obj      \\ \hline
\endfirsthead
\hline
\multirow{2}{*}{$(m,n,p)$} & \multicolumn{2}{c|}{ProxSSN} & \multicolumn{2}{c|}{ProxGD} & \multicolumn{2}{c|}{AManPG} & \multicolumn{2}{c|}{ARPG}    \\ \cline{2-9}
     & time   &  obj  & time & obj   & time & obj   & time & obj      \\ \hline
\endhead
\hline
\endfoot

 100 / 500 /  10 &  1.58 & 1.28380   &  13.59 & 1.28380  &  1.99 & 1.28380  &  1.75 & 1.28380  \\ \hline
 100 / 500 /  15 &  1.16 & 1.85986   &  7.05 & 1.85986  &  1.82 & 1.85986  &  1.696 & 1.85986  \\ \hline
 100 / 500 /  20 &  1.71 & 2.44963   &  16.59 & 2.44963  &  3.40 & 2.44963  &  2.96 & 2.44963  \\ \hline
 100 / 500 /  25 &  2.36 & 3.00555   &  15.66 & 3.00555  &  4.05 & 3.00555  &  3.97 & 3.00555  \\ \hline
 100 / 500 /  30 &  0.84 & 3.58139   &  12.15 & 3.58139  &  3.21 & 3.58139  &  3.16 & 3.58139  \\ \hline
 100 / 600 /  10 &  0.40 & 1.39524   &  2.47 & 1.39524  &  0.53 & 1.39524  &  0.62 & 1.39524  \\ \hline
 100 / 600 /  15 &  0.19 & 2.04237   &  1.45 & 2.04237  &  0.51 & 2.04237  &  0.48 & 2.04237  \\ \hline
 100 / 600 /  20 &  5.98 & 2.68717   &  28.37 & 2.68875  &  7.53 & 2.68717  &  9.03 & 2.68717  \\ \hline
 100 / 600 /  25 &  0.34 & 3.31583   &  2.96 & 3.31583  &  1.27 & 3.31583  &  1.11 & 3.31583  \\ \hline
 100 / 600 /  30 &  2.47 & 3.93575   &  35.16 & 3.93579  &  9.68 & 3.93580  &  10.49 & 3.93580  \\ \hline
 100 / 700 /  10 &  1.50 & 1.50657   &  8.51 & 1.50657  &  1.65 & 1.50657  &  1.68 & 1.50657  \\ \hline
 100 / 700 /  15 &  0.60 & 2.21769   &  2.61 & 2.21769  &  0.80 & 2.21769  &  0.84 & 2.21769  \\ \hline
 100 / 700 /  20 &  2.02 & 2.92664   &  19.00 & 2.92664  &  3.42 & 2.92664  &  3.20 & 2.92664  \\ \hline
 100 / 700 /  25 &  2.22 & 3.59936   &  21.64 & 3.59936  &  5.04 & 3.59936  &  4.55 & 3.59936  \\ \hline
 100 / 700 /  30 &  2.44 & 4.23529   &  42.76 & 4.23540  &  6.04 & 4.23529  &  5.07 & 4.23529  \\ \hline
 100 / 800 /  10 &  0.27 & 1.60610   &  1.64 & 1.60610  &  0.45 & 1.60610  &  0.53 & 1.60610  \\ \hline
 100 / 800 /  15 &  0.46 & 2.36806   &  4.67 & 2.36806  &  0.87 & 2.36806  &  0.91 & 2.36806  \\ \hline
 100 / 800 /  20 &  1.64 & 3.09902   &  16.86 & 3.09902  &  3.89 & 3.09902  &  4.03 & 3.09902  \\ \hline
 100 / 800 /  25 &  1.38 & 3.82806   &  19.20 & 3.82806  &  4.08 & 3.82806  &  3.98 & 3.82806  \\ \hline
 100 / 800 /  30 &  5.77 & 4.55643   &  41.61 & 4.55681  &  13.49 & 4.55644  &  12.97 & 4.55644  \\ \hline
 100 / 900 /  10 &  0.76 & 1.71069   &  4.21 & 1.71069  &  0.80 & 1.71069  &  0.90 & 1.71069  \\ \hline
 100 / 900 /  15 &  0.28 & 2.51949   &  2.68 & 2.51949  &  0.93 & 2.51949  &  0.88 & 2.51949  \\ \hline
 100 / 900 /  20 &  1.44 & 3.28293   &  10.74 & 3.28294  &  2.63 & 3.28294  &  2.50 & 3.28294  \\ \hline
 100 / 900 /  25 &  1.66 & 4.09218   &  22.07 & 4.09218  &  6.50 & 4.09218  &  6.23 & 4.09218  \\ \hline
 100 / 900 /  30 &  3.70 & 4.81562   &  38.57 & 4.81896  &  13.25 & 4.81563  &  13.80 & 4.81563  \\ \hline
 100 / 1000 /  10 &  1.42 & 1.80718   &  10.43 & 1.80718  &  1.81 & 1.80718  &  1.69 & 1.80718  \\ \hline
 100 / 1000 /  15 &  2.38 & 2.64274   &  19.57 & 2.64274  &  3.65 & 2.64274  &  3.45 & 2.64274  \\ \hline
 100 / 1000 /  20 &  0.55 & 3.47447   &  6.46 & 3.47447  &  1.92 & 3.47447  &  1.94 & 3.47447  \\ \hline
 100 / 1000 /  25 &  2.23 & 4.25629   &  29.63 & 4.25629  &  6.83 & 4.25629  &  6.84 & 4.25629  \\ \hline
 100 / 1000 /  30 &  5.37 & 5.08015   &  44.92 & 5.08103  &  17.01 & 5.08015  &  15.79 & 5.08015  \\ \hline

\end{longtable}
\end{footnotesize}

\subsection{Sparse least square regression}
In this subsection, we consider the sparse least-square problem \eqref{prob:slr1}, which can be regarded as a nonsmooth problem on the oblique manifold. We test the same algorithms as in subsection \ref{sec:spca} for the comparisons. All parameters and strategies follow the setup discussed in the last subsection except  $\mathrm{tol} = 10^{-10}nm$.  The numerical results are presented in Figures \ref{fig:perf_lsr_time}-\ref{fig:profile_lsr_time}. In general,
the overall performance of different methods is similar to the results shown in the last subsection. It is clear that ProxSSN is the fastest method for solving problem \eqref{prob:slr1}, both in terms of the objective function value and the wall-clock time. Table \ref{tab:lsr} shows the detailed results for different combinations of $m,n$. We see that ProxSSN compares favorably with the other algorithms and outperforms the first-order algorithm ProxGD.
\begin{figure}[!htb]
\centering
\subfigure{
\includegraphics[width=0.45\textwidth]{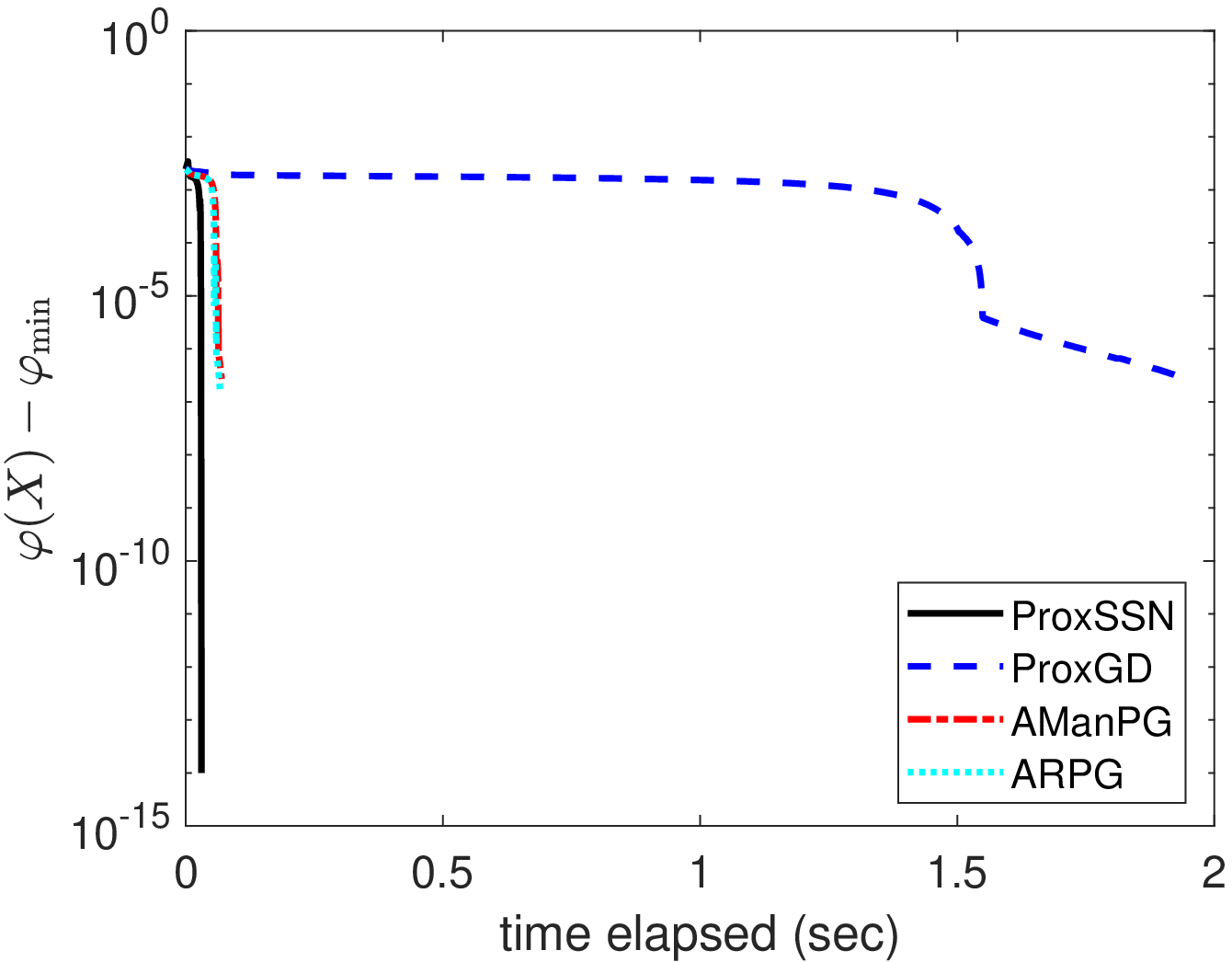}}
\subfigure{
\includegraphics[width=0.45\textwidth]{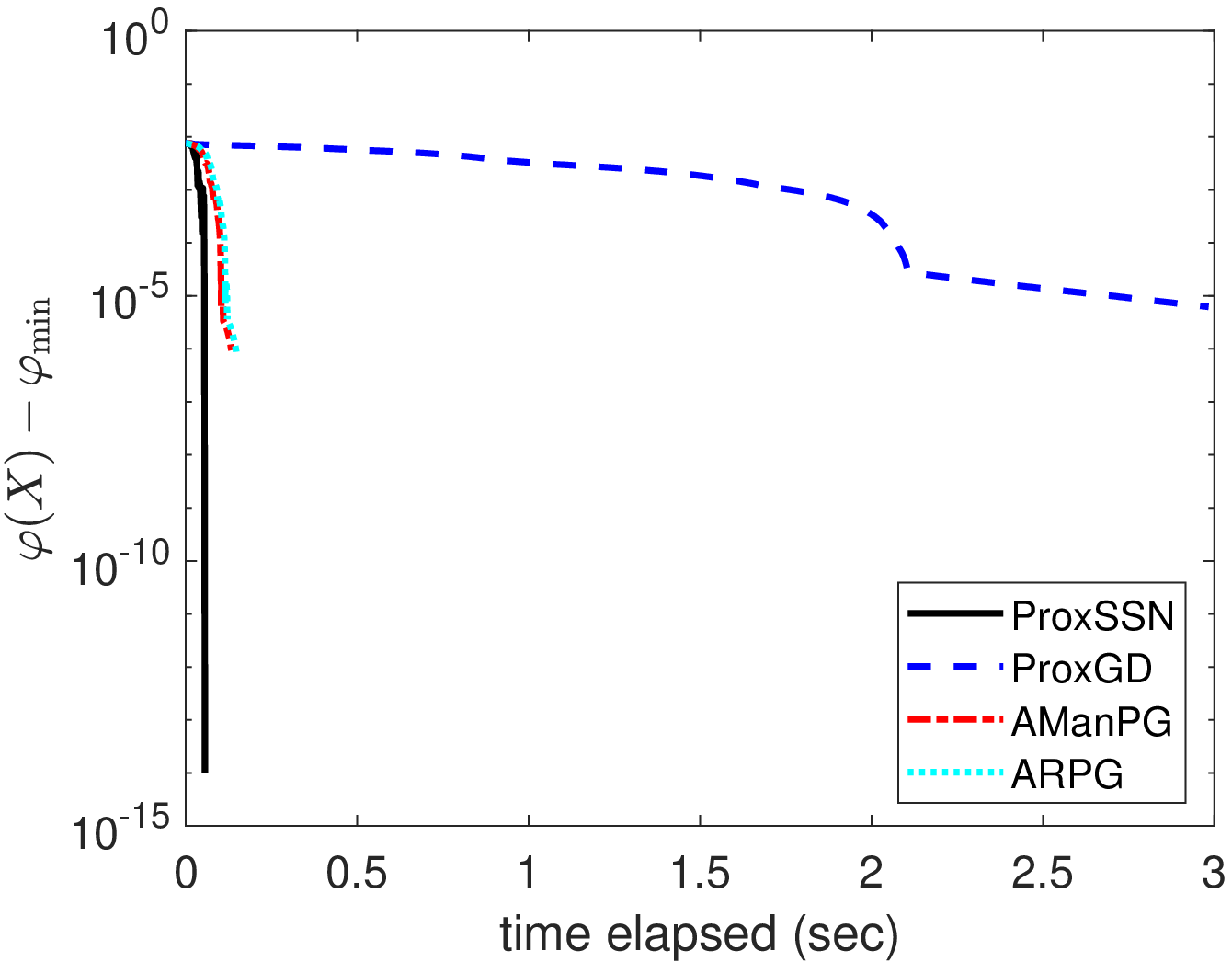}}
\caption{The trajectories of the objective function values with respect to the wall-clock time on the sparse least square regression \eqref{prob:slr1} with $m = 20,\lambda = 0.01$. Left: $n=2000$; right: $n=3000$.}\label{fig:perf_lsr_time}
\end{figure}

\begin{figure}[!htb]
\centering
\subfigure{
\includegraphics[width=0.45\textwidth]{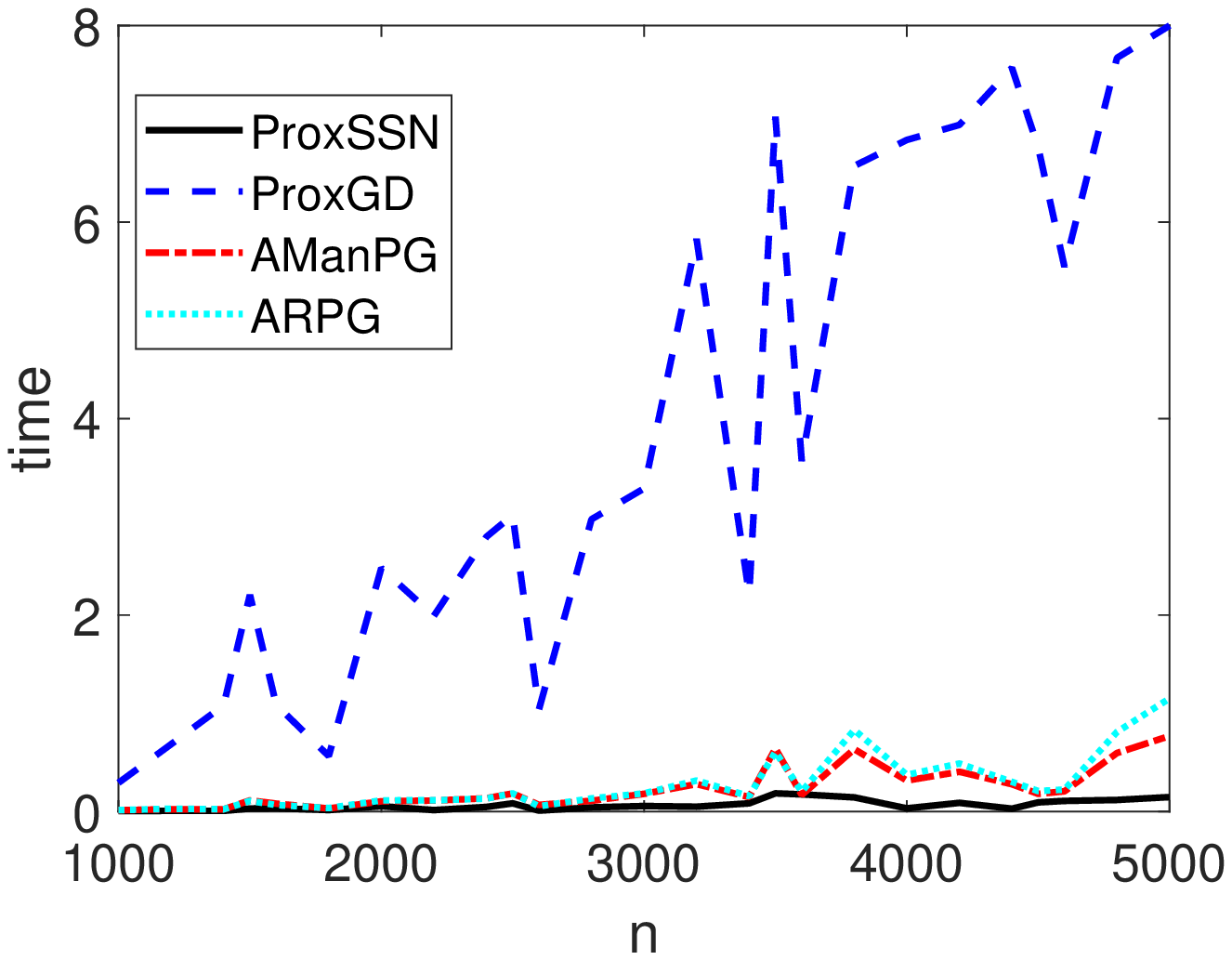}}
\subfigure{
\includegraphics[width=0.45\textwidth]{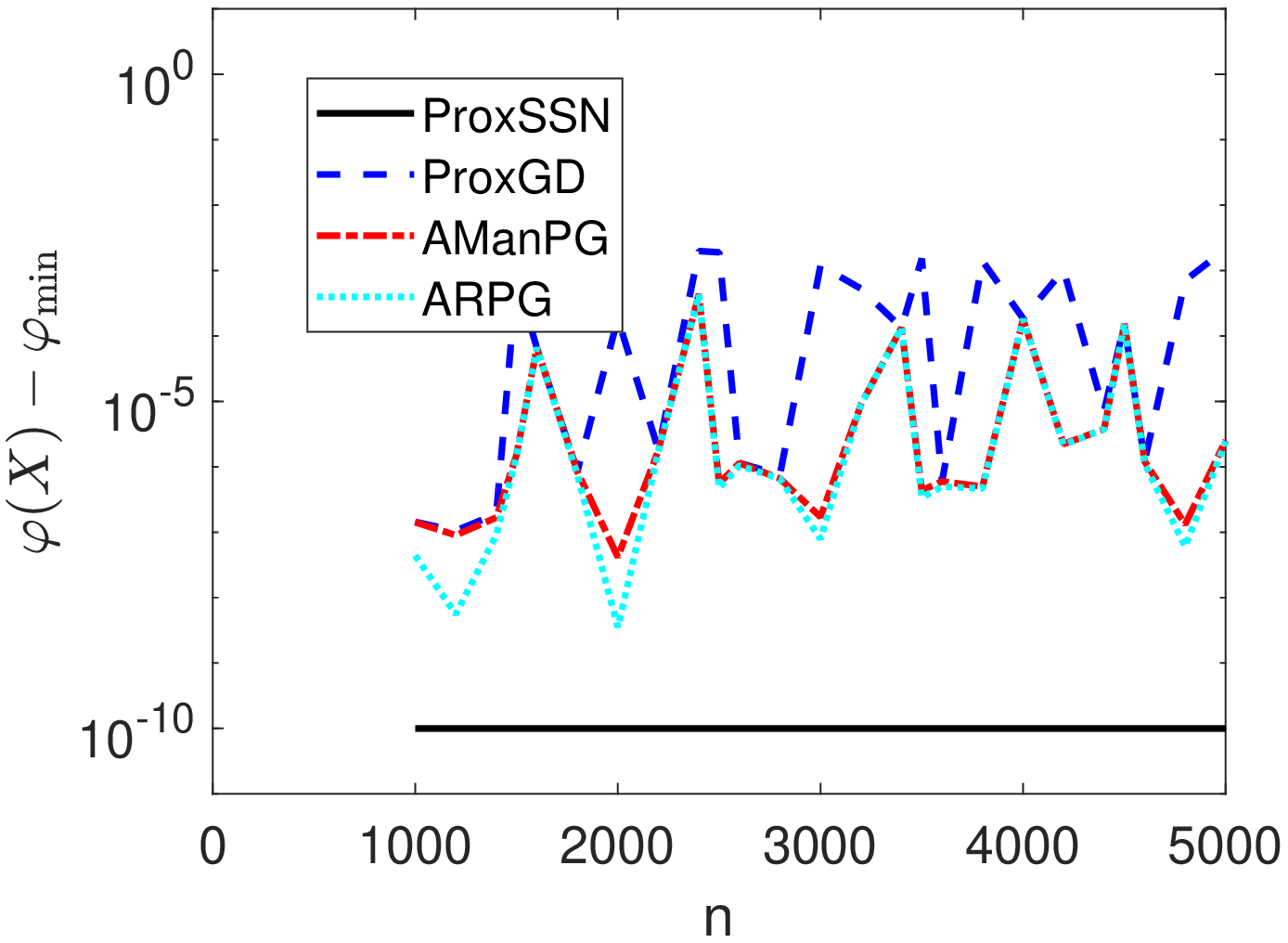}}
\caption{Comparisons of wall-clock time and the objective function values on the sparse least square regression \eqref{prob:slr1} with $m = 30,\lambda = 0.01$ for different $n$.}\label{fig:perf_lsr_m}
\end{figure}

\begin{footnotesize}
\setlength{\tabcolsep}{4pt}
%\footnotesize
\centering
\begin{longtable}{|c|cc|cc|cc|cc|}
\caption{Computational results of least square regression }\label{tab:lsr}\\
\hline
\multirow{2}{*}{$(m,n)$} & \multicolumn{2}{c|}{ProxSSN} & \multicolumn{2}{c|}{ProxGD} & \multicolumn{2}{c|}{AManPG} & \multicolumn{2}{c|}{ARPG}    \\ \cline{2-9}
     & time   &  obj  & time & obj   & time & obj   & time & obj      \\ \hline
\endfirsthead
\hline
\multirow{2}{*}{$(m,n)$} & \multicolumn{2}{c|}{ProxSSN} & \multicolumn{2}{c|}{ProxGD} & \multicolumn{2}{c|}{AManPG} & \multicolumn{2}{c|}{ARPG}    \\ \cline{2-9}
     & time   &  obj  & time & obj   & time & obj   & time & obj      \\ \hline
\endhead
\hline
\endfoot

 20 / 3000  &  0.04 & 3.43796e-02  &  2.92 & 3.43839e-02 &  0.07 & 3.43799e-02 &  0.07 & 3.43798e-02 \\ \hline
 20 / 3200  &  0.18 & 3.39873e-02  &  1.07 & 3.39886e-02 &  0.08 & 3.39886e-02 &  0.09 & 3.39885e-02 \\ \hline
 20 / 3400  &  0.11 & 3.23110e-02  &  5.61 & 3.23240e-02 &  0.26 & 3.23123e-02 &  0.29 & 3.23122e-02 \\ \hline
 20 / 3600  &  0.17 & 3.17896e-02  &  5.18 & 3.20365e-02 &  0.24 & 3.20365e-02 &  0.27 & 3.20364e-02 \\ \hline
 20 / 3800  &  0.05 & 3.43032e-02  &  5.94 & 3.43061e-02 &  0.24 & 3.43048e-02 &  0.27 & 3.43047e-02 \\ \hline
 20 / 4000  &  0.09 & 3.44652e-02  &  6.07 & 3.54900e-02 &  0.36 & 3.44664e-02 &  0.42 & 3.44663e-02 \\ \hline
 20 / 4200  &  0.21 & 3.60764e-02  &  6.34 & 3.67852e-02 &  0.43 & 3.60786e-02 &  0.50 & 3.60785e-02 \\ \hline
 20 / 4400  &  0.11 & 3.36402e-02  &  6.49 & 3.68569e-02 &  0.60 & 3.36402e-02 &  0.80 & 3.36402e-02 \\ \hline
 20 / 4600  &  0.13 & 3.39844e-02  &  6.59 & 3.71441e-02 &  0.92 & 3.35500e-02 &  1.32 & 3.35500e-02 \\ \hline
 20 / 4800  &  0.16 & 3.40047e-02  &  6.90 & 3.40144e-02 &  0.34 & 3.40059e-02 &  0.42 & 3.40058e-02 \\ \hline
 20 / 5000  &  0.06 & 3.32278e-02  &  6.90 & 3.54494e-02 &  0.86 & 3.32286e-02 &  1.22 & 3.32285e-02 \\ \hline
 30 / 3000  &  0.05 & 3.73733e-02  &  3.28 & 3.85623e-02 &  0.18 & 3.73735e-02 &  0.18 & 3.73734e-02 \\ \hline
 30 / 3200  &  0.05 & 3.46184e-02  &  5.82 & 3.51461e-02 &  0.28 & 3.46273e-02 &  0.31 & 3.46273e-02 \\ \hline
 30 / 3400  &  0.08 & 3.57899e-02  &  2.21 & 3.59235e-02 &  0.14 & 3.59235e-02 &  0.15 & 3.59234e-02 \\ \hline
 30 / 3600  &  0.17 & 3.73116e-02  &  3.56 & 3.73122e-02 &  0.17 & 3.73122e-02 &  0.20 & 3.73121e-02 \\ \hline
 30 / 3800  &  0.14 & 3.76258e-02  &  6.57 & 3.90207e-02 &  0.63 & 3.76263e-02 &  0.83 & 3.76263e-02 \\ \hline
 30 / 4000  &  0.03 & 4.06294e-02  &  6.83 & 4.08145e-02 &  0.31 & 4.08106e-02 &  0.37 & 4.08105e-02 \\ \hline
 30 / 4200  &  0.08 & 3.96908e-02  &  6.98 & 4.07081e-02 &  0.40 & 3.96931e-02 &  0.48 & 3.96930e-02 \\ \hline
 30 / 4400  &  0.03 & 3.95462e-02  &  7.57 & 3.95534e-02 &  0.27 & 3.95500e-02 &  0.30 & 3.95500e-02 \\ \hline
 30 / 4600  &  0.10 & 3.55181e-02  &  5.54 & 3.55193e-02 &  0.20 & 3.55193e-02 &  0.22 & 3.55192e-02 \\ \hline
 30 / 4800  &  0.11 & 3.85425e-02  &  7.67 & 3.92473e-02 &  0.59 & 3.85426e-02 &  0.81 & 3.85425e-02 \\ \hline
 30 / 5000  &  0.14 & 3.95414e-02  &  8.00 & 4.16688e-02 &  0.77 & 3.95439e-02 &  1.15 & 3.95438e-02 \\ \hline
 50 / 3000  &  0.05 & 4.14906e-02  &  4.46 & 4.16952e-02 &  0.23 & 4.14908e-02 &  0.24 & 4.14907e-02 \\ \hline
 50 / 3200  &  0.03 & 4.08372e-02  &  2.60 & 4.08408e-02 &  0.17 & 4.08407e-02 &  0.18 & 4.08407e-02 \\ \hline
 50 / 3400  &  0.12 & 4.53565e-02  &  5.64 & 4.58502e-02 &  0.18 & 4.58502e-02 &  0.19 & 4.58501e-02 \\ \hline
 50 / 3600  &  0.05 & 4.52462e-02  &  8.17 & 4.57722e-02 &  0.35 & 4.52464e-02 &  0.43 & 4.52463e-02 \\ \hline
 50 / 3800  &  0.06 & 4.12851e-02  &  3.47 & 4.12852e-02 &  0.19 & 4.12852e-02 &  0.23 & 4.12851e-02 \\ \hline
 50 / 4000  &  0.08 & 4.44167e-02  &  10.12 & 4.40984e-02 &  0.82 & 4.40979e-02 &  0.40 & 4.40983e-02 \\ \hline
 50 / 4200  &  0.23 & 4.16107e-02  &  10.48 & 4.16618e-02 &  0.60 & 4.16623e-02 &  1.26 & 4.16622e-02 \\ \hline
 50 / 4400  &  0.13 & 4.37490e-02  &  13.72 & 4.37491e-02 &  0.62 & 4.37491e-02 &  1.13 & 4.37490e-02 \\ \hline
 50 / 4600  &  0.07 & 4.41428e-02  &  3.45 & 4.41463e-02 &  0.31 & 4.41463e-02 &  0.42 & 4.41462e-02 \\ \hline
 50 / 4800  &  0.44 & 4.40181e-02  &  13.08 & 4.49775e-02 &  0.96 & 4.40813e-02 &  1.51 & 4.40812e-02 \\ \hline
 50 / 5000  &  0.18 & 4.02113e-02  &  13.16 & 4.38594e-02 &  0.91 & 4.05313e-02 &  1.35 & 4.05312e-02 \\ \hline

\end{longtable}
\end{footnotesize}

\begin{figure}[!htb]
\centering
\subfigure{
\includegraphics[width=0.45\textwidth]{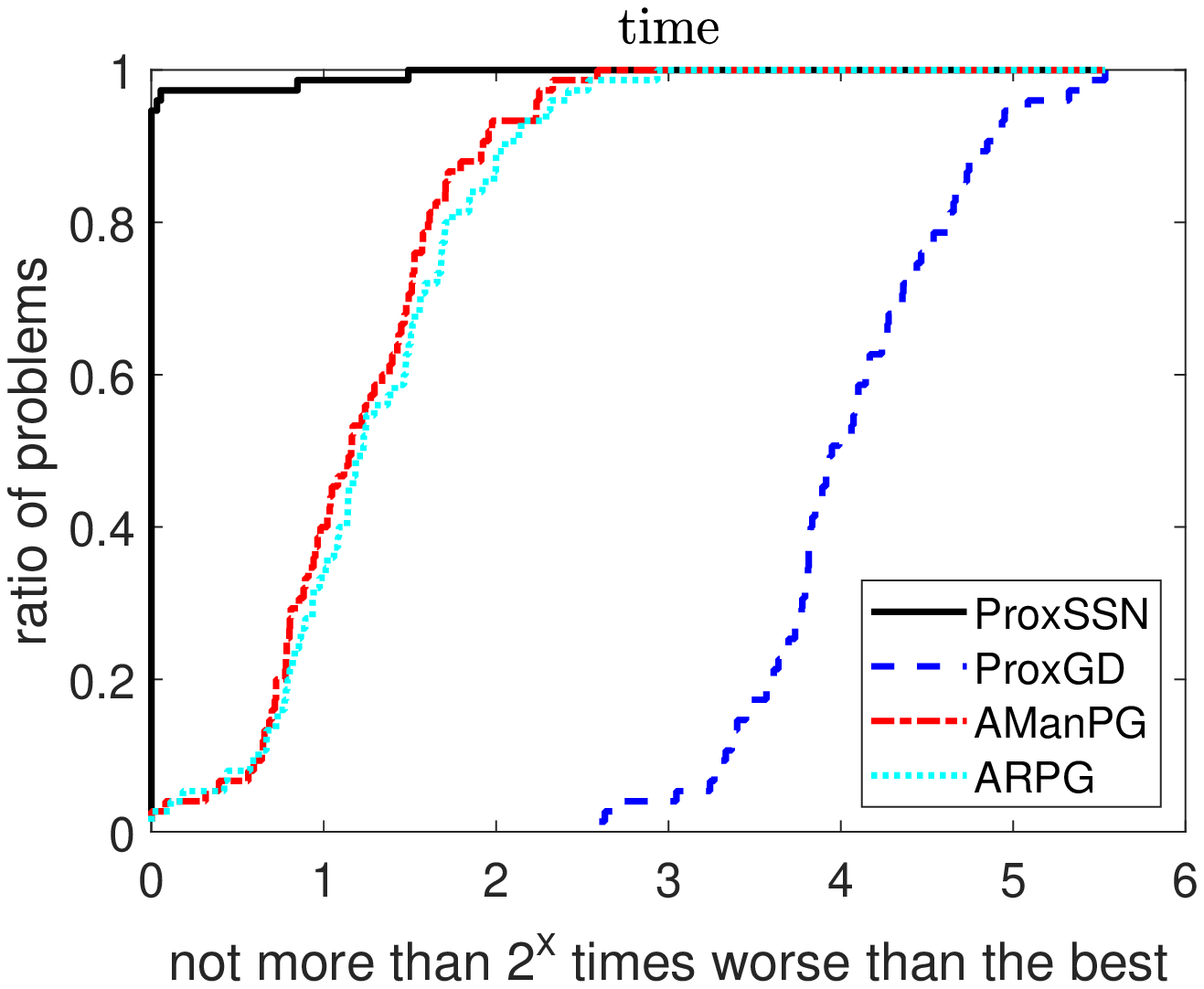}}
\subfigure{
\includegraphics[width=0.45\textwidth]{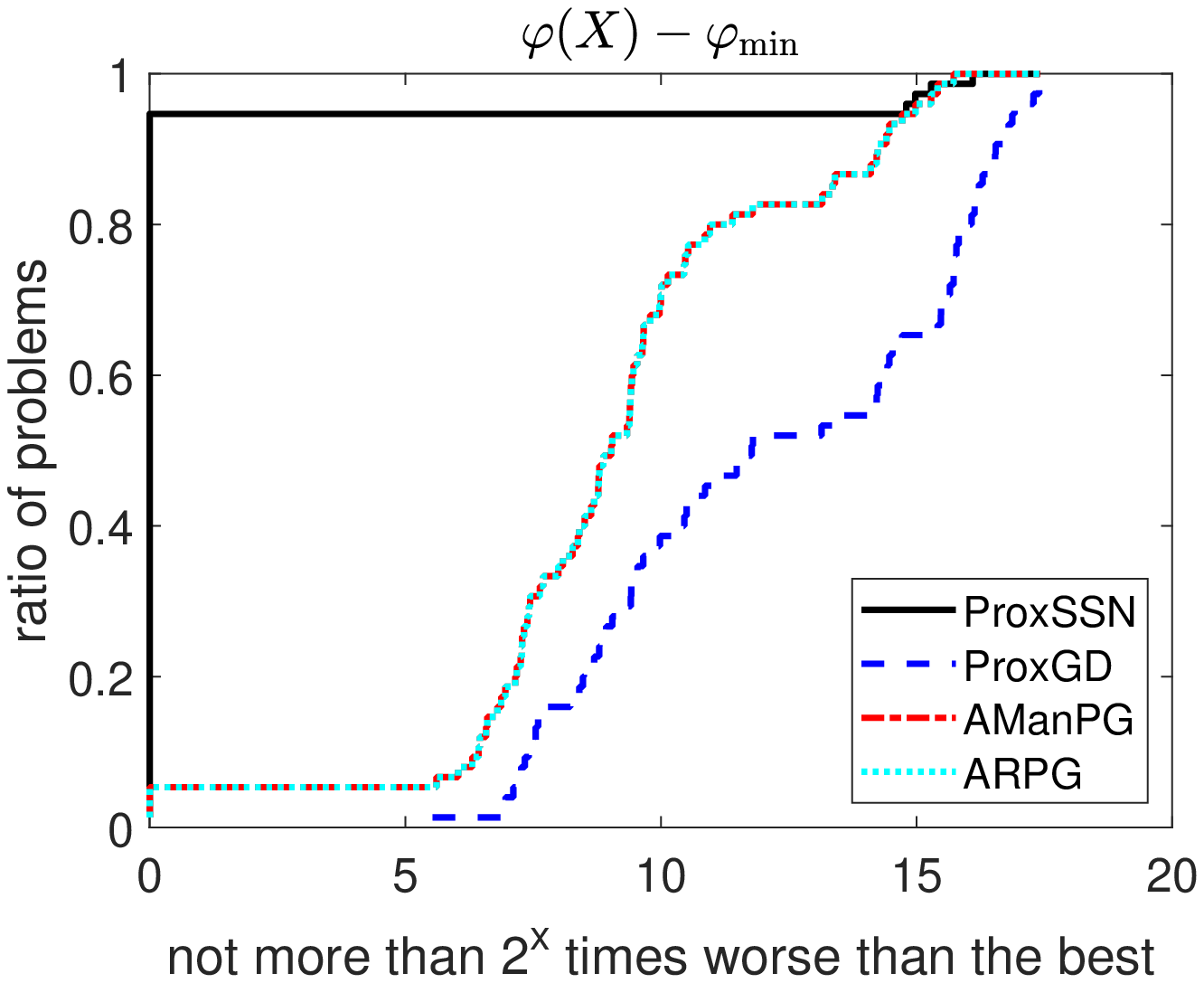}}
\caption{The performance profiles on the sparse least square regression \eqref{prob:slr1}.}\label{fig:profile_lsr_time}
\end{figure}

\subsection{Nonnegative principal component analysis}
In this subsection, we consider the nonnegative PCA model \eqref{prob:npca} on the oblique manifold.
% \be \label{prob:npca} \min_{X \in {\rm Ob}(n,p)} \;\; \|X^\top A^\top A X - D^2 \|_F^2, ~~\mathrm{s.t.} ~~ X\geq 0,  \ee
% 	where ${\rm Ob}(n,p) = \{ X \in \R^{n \times p}:\diag(X^\top X) = 1_{p} \}$, $D$ is the diagonal matrix whose diagonal entries are the first $p$ dominant singular values of $A$.
	All parameters of our algorithm are the same as those in subsection \ref{sec:spca}. Since AManPG and ARPG cannot achieve our requirement for accuracy in most testing cases, we omit them in this experiment. The possible reason is that the convergence of AManPG and ARPG relies on the Lipschitz continuity of the nonsmooth part, while it is not the case for the indicator function of $\delta_{X\geq 0}$.
	Hence, we only compare our algorithm with ProxGD. The comparisons are illustrated in Figures \ref{fig:perf_npca_m} and \ref{fig:profile_npca_time} and Table \ref{tab:npca} for the computational results. Those results show that ProxSSN achieves better results and converges much faster to highly accurate solutions compared with ProxGD.

\begin{figure}[!htb]
\centering
\subfigure{
\includegraphics[width=0.45\textwidth]{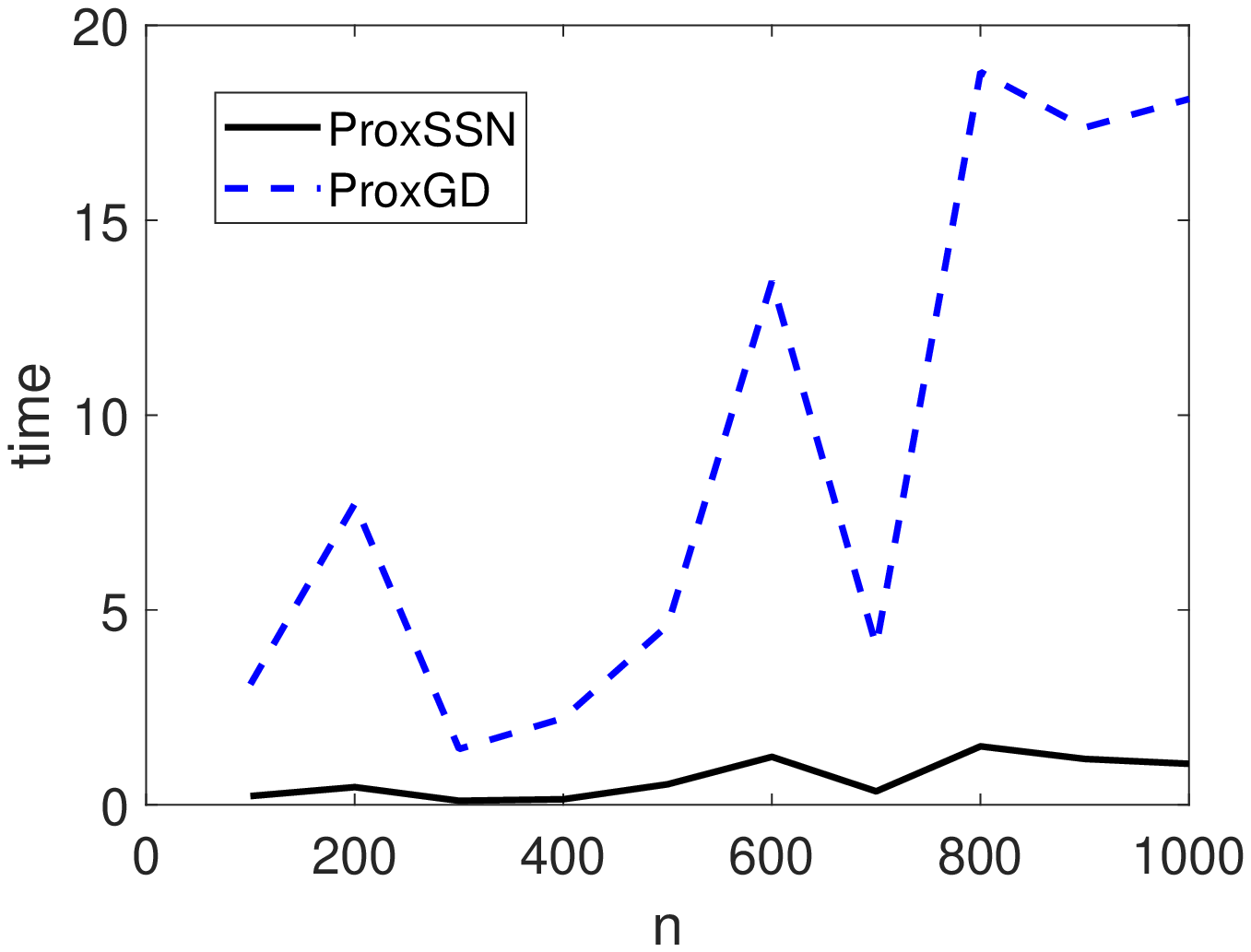}}
\subfigure{
\includegraphics[width=0.45\textwidth]{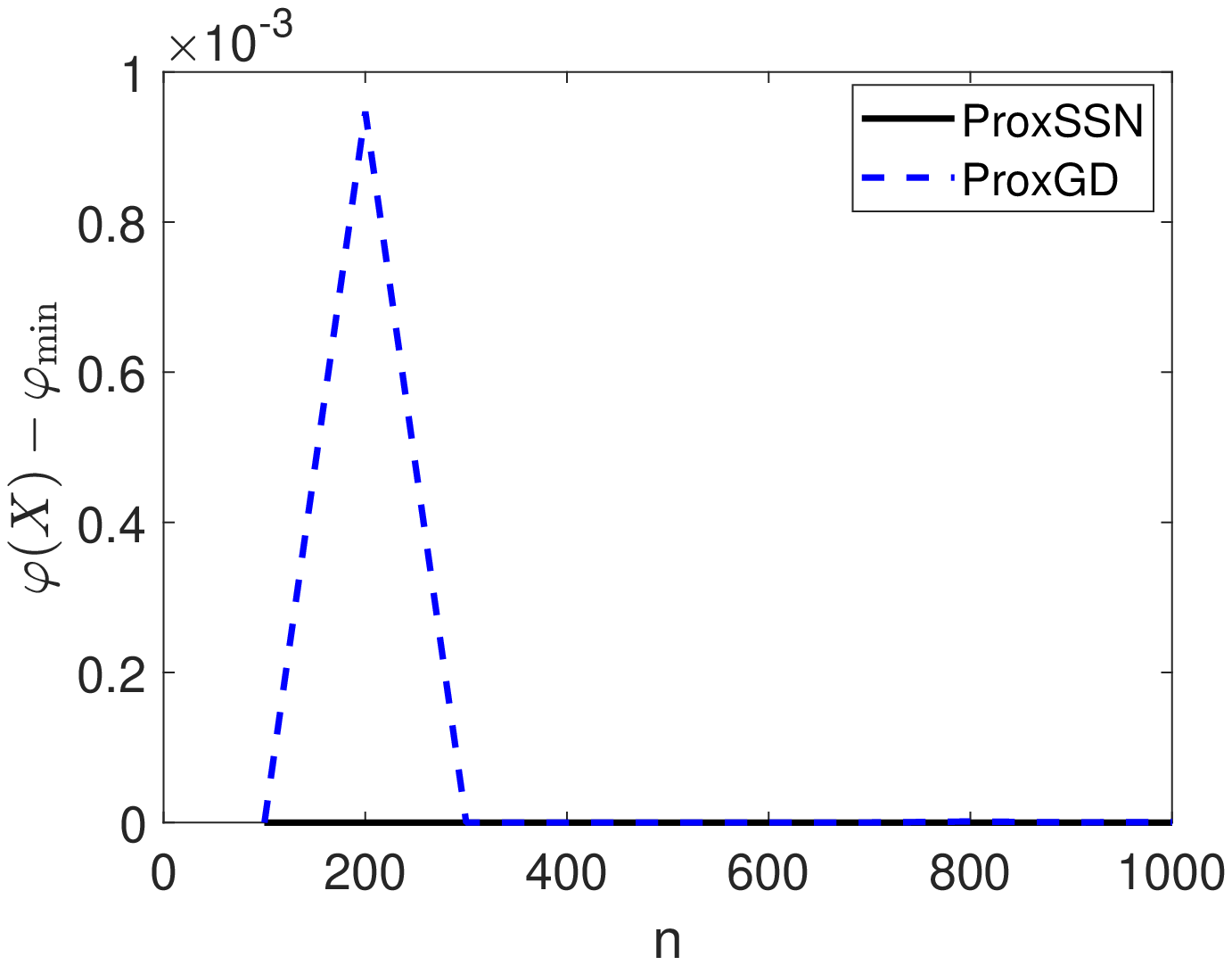}}
\caption{Comparisons of wall-clock time and the objective function values on the nonnegative PCA problem \eqref{prob:npca} with $p = 20$ for different $n$.}\label{fig:perf_npca_m}
\end{figure}

\begin{figure}[!htb]
\centering
\subfigure{
\includegraphics[width=0.45\textwidth]{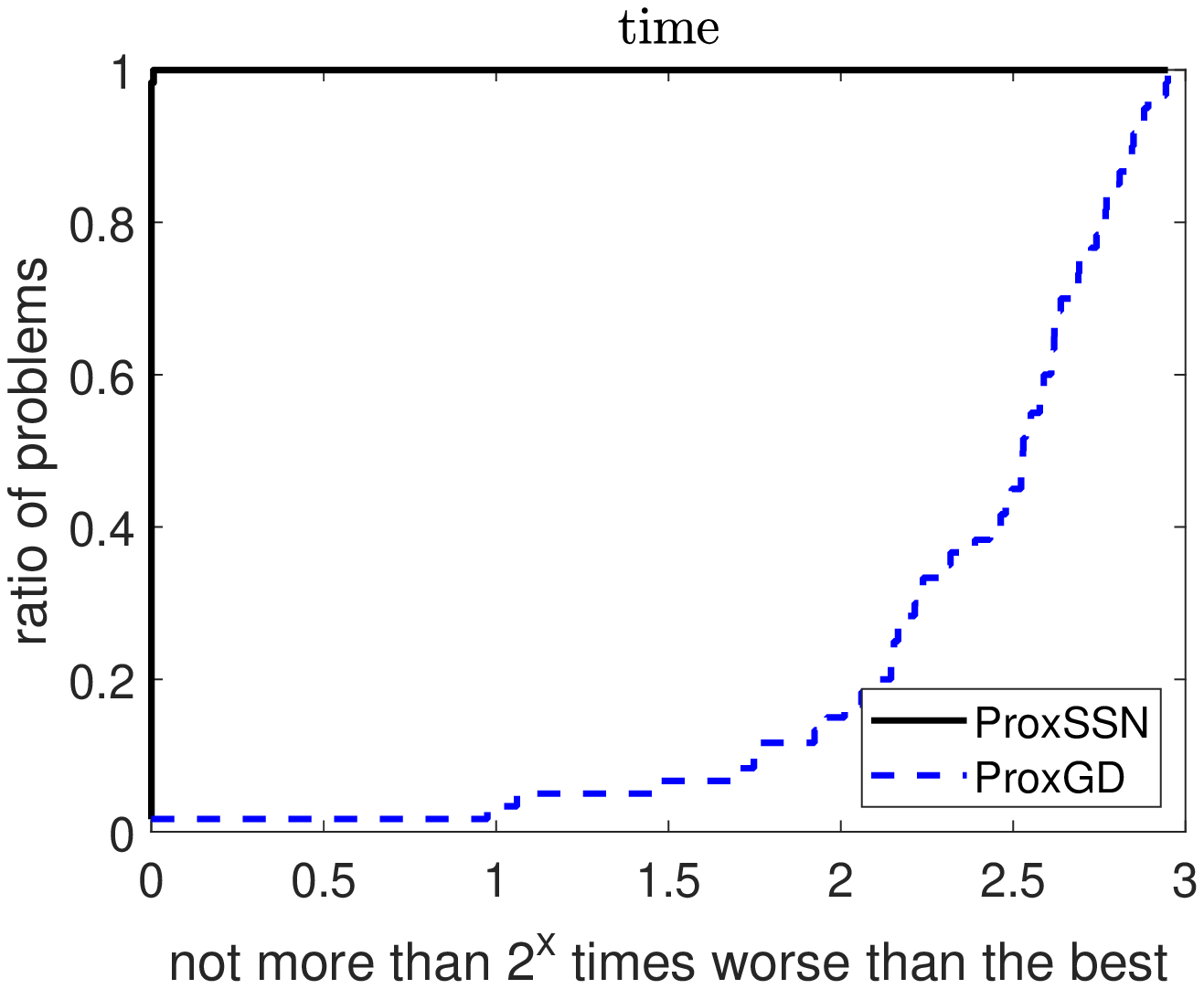}}
\subfigure{
\includegraphics[width=0.45\textwidth]{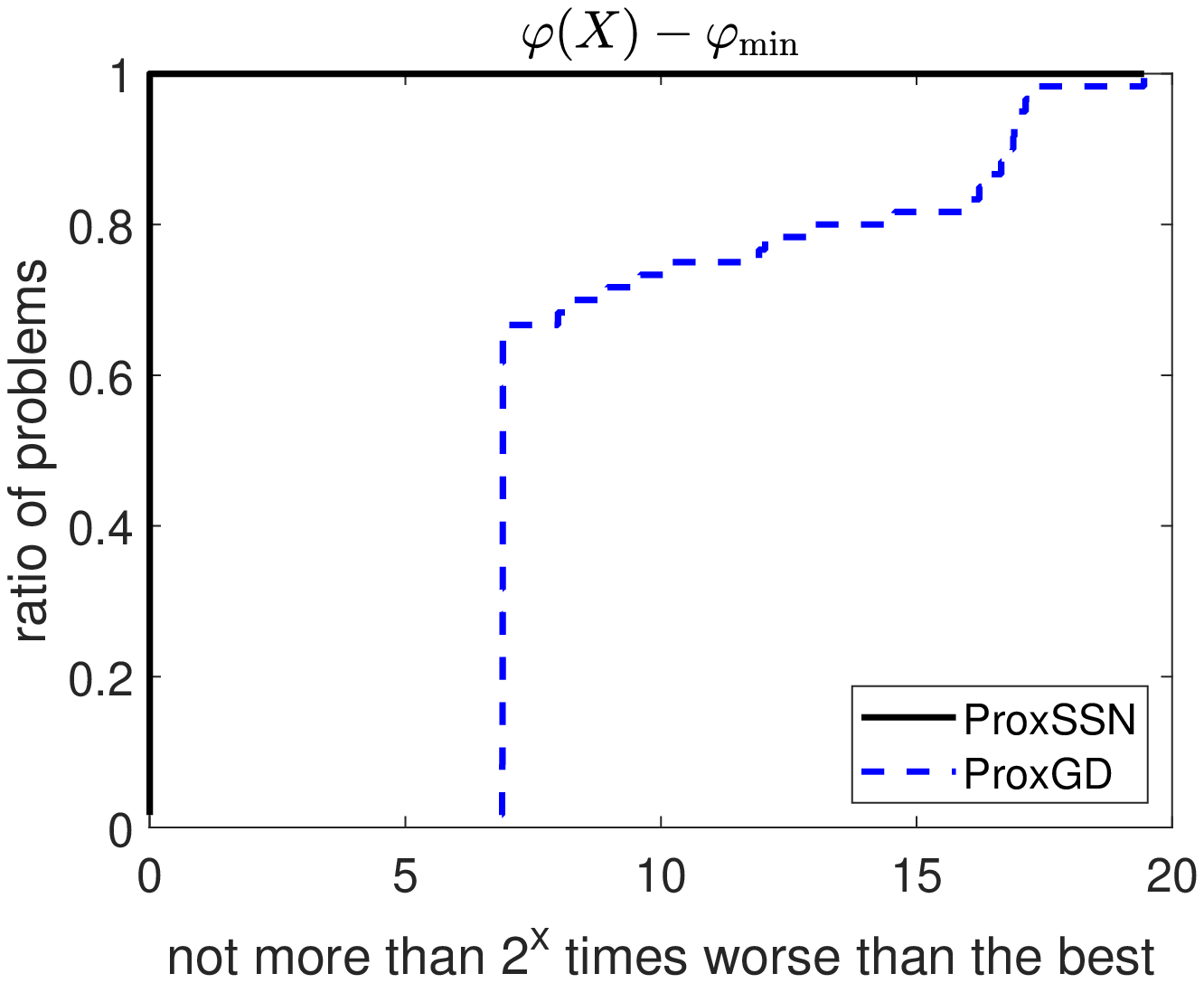}}
\caption{The performance profiles on nonnegative PCA problem \eqref{prob:npca}.}\label{fig:profile_npca_time}
\end{figure}

\begin{footnotesize}
%\setlength{\tabcolsep}{1.2pt}
%\footnotesize
\centering
\begin{longtable}{|c|cccc|cccc|}
\caption{Computational results of the nonnegative PCA problem \eqref{prob:npca}. }\label{tab:npca}\\
\hline
\multirow{2}{*}{$(n,p)$} & \multicolumn{4}{c|}{ProxSSN} & \multicolumn{4}{c|}{ProxGD}      \\ \cline{2-9}
     & time   &  obj  & err  & iter  & time   &  obj  & err  & iter        \\ \hline
\endfirsthead
\hline
\multirow{2}{*}{$(n,p)$} & \multicolumn{4}{c|}{ProxSSN} & \multicolumn{4}{c|}{ProxGD}      \\ \cline{2-9}
     & time   &  obj  & err  & iter  & time   &  obj  & err  & iter        \\ \hline
\endhead
\hline
\endfoot

500 / 10  & 0.35  & 1.166866 & 1.23e-7  & 66 (9.2) & 2.46  & 1.166866 & 2.93e-5  & 2840  \\ \hline
500 / 15  & 0.22  & 1.619850 & 6.55e-7  & 33 (9.5) & 1.96  & 1.619850 & 2.60e-5  & 1838  \\ \hline
500 / 20  & 0.53  & 1.942255 & 8.52e-7  & 64 (9.8) & 4.60  & 1.942255 & 2.29e-5  & 3413  \\ \hline
500 / 25  & 0.72  & 2.300220 & 7.04e-7  & 73 (9.8) & 10.42  & 2.300220 & 1.70e-5  & 7317  \\ \hline
500 / 30  & 0.89  & 2.523960 & 6.56e-7  & 83 (9.9) & 15.41  & 2.523999 & 2.69e-4  & 10000  \\ \hline
500 / 5  & 0.04  & 0.592243 & 8.56e-8  & 13 (8.8) & 0.13  & 0.592244 & 4.83e-5  & 189  \\ \hline
600 / 10  & 0.12  & 1.223420 & 5.71e-7  & 20 (9.4) & 1.32  & 1.223420 & 2.57e-5  & 1362  \\ \hline
600 / 15  & 0.34  & 1.894680 & 3.86e-7  & 47 (9.7) & 5.43  & 1.894680 & 2.44e-5  & 4455  \\ \hline
600 / 20  & 1.23  & 2.261036 & 1.19e-6  & 130 (9.6) & 13.43  & 2.261036 & 1.63e-5  & 9617  \\ \hline
600 / 25  & 0.90  & 2.238704 & 9.12e-7  & 91 (9.8) & 14.71  & 2.241462 & 1.54e-3  & 10000  \\ \hline
600 / 30  & 0.92  & 2.510240 & 1.19e-6  & 94 (9.8) & 16.34  & 2.510453 & 1.81e-5  & 10000  \\ \hline
600 / 5  & 0.04  & 0.782584 & 8.94e-8  & 12 (9.2) & 0.52  & 0.782584 & 2.43e-5  & 757  \\ \hline
700 / 10  & 0.48  & 1.332547 & 3.06e-7  & 66 (9.5) & 4.74  & 1.332547 & 2.90e-5  & 5031  \\ \hline
700 / 15  & 0.23  & 1.891921 & 4.99e-7  & 32 (9.7) & 3.09  & 1.891922 & 1.61e-5  & 2448  \\ \hline
700 / 20  & 0.35  & 2.232710 & 7.31e-7  & 38 (9.7) & 4.07  & 2.232710 & 1.86e-5  & 2787  \\ \hline
700 / 25  & 1.00  & 2.578730 & 1.39e-6  & 90 (9.9) & 18.99  & 2.578745 & 1.61e-4  & 10000  \\ \hline
700 / 30  & 2.01  & 2.997021 & 1.90e-6  & 124 (9.8) & 14.99  & 3.025005 & 1.36e-5  & 6748  \\ \hline
700 / 5  & 0.09  & 0.751121 & 1.23e-7  & 19 (9.1) & 1.06  & 0.751121 & 4.40e-5  & 1475  \\ \hline
800 / 10  & 0.62  & 1.361048 & 6.58e-7  & 57 (9.5) & 5.33  & 1.361048 & 2.53e-5  & 3805  \\ \hline
800 / 15  & 1.07  & 1.837726 & 5.13e-7  & 99 (9.7) & 17.68  & 1.839436 & 7.48e-4  & 10000  \\ \hline
800 / 20  & 1.50  & 2.262145 & 1.20e-6  & 115 (9.5) & 18.80  & 2.262147 & 5.92e-5  & 10000  \\ \hline
800 / 25  & 2.30  & 2.621645 & 1.51e-6  & 158 (9.8) & 19.63  & 2.623857 & 1.42e-4  & 10000  \\ \hline
800 / 30  & 1.58  & 2.943294 & 2.20e-6  & 122 (9.7) & 19.78  & 2.944495 & 1.71e-3  & 10000  \\ \hline
800 / 5  & 0.07  & 0.754357 & 1.28e-7  & 10 (9.0) & 0.31  & 0.754357 & 4.35e-5  & 257  \\ \hline
900 / 10  & 0.10  & 1.374185 & 3.69e-7  & 14 (9.3) & 1.51  & 1.374185 & 2.40e-5  & 1200  \\ \hline
900 / 15  & 0.86  & 1.933525 & 1.06e-6  & 93 (9.7) & 7.89  & 1.933525 & 1.07e-5  & 5513  \\ \hline
900 / 20  & 1.18  & 2.360027 & 1.36e-6  & 107 (9.7) & 17.38  & 2.360027 & 1.74e-5  & 10000  \\ \hline
900 / 25  & 1.88  & 2.773641 & 1.69e-6  & 153 (9.8) & 19.07  & 2.777065 & 3.41e-4  & 10000  \\ \hline
900 / 30  & 1.60  & 3.157731 & 2.02e-6  & 121 (9.8) & 21.05  & 3.159992 & 3.55e-3  & 10000  \\ \hline
900 / 5  & 0.27  & 0.770418 & 2.27e-7  & 62 (7.9) & 1.54  & 0.770418 & 2.59e-5  & 1672  \\ \hline
1000 / 10  & 0.64  & 1.376750 & 2.84e-7  & 81 (9.0) & 3.50  & 1.376750 & 3.47e-5  & 2719  \\ \hline
1000 / 15  & 0.19  & 2.049750 & 1.08e-6  & 21 (9.5) & 2.65  & 2.049750 & 2.16e-5  & 1673  \\ \hline
1000 / 20  & 1.05  & 2.581317 & 1.41e-6  & 85 (9.7) & 18.11  & 2.581318 & 2.08e-5  & 10000  \\ \hline
1000 / 25  & 1.30  & 3.043254 & 1.18e-6  & 101 (9.8) & 20.56  & 3.045420 & 4.82e-4  & 10000  \\ \hline
1000 / 30  & 1.79  & 3.516861 & 2.84e-6  & 129 (9.8) & 23.79  & 3.517976 & 1.25e-3  & 10000  \\ \hline
1000 / 5  & 0.05  & 0.804861 & 2.75e-7  & 10 (9.0) & 0.33  & 0.804861 & 3.53e-5  & 390  \\ \hline

\end{longtable}
\end{footnotesize}

\subsection{Bose-Einstein condensates}
In this subsection, we consider the Bose-Einstein condensates (BEC) problem \cite{aftalion2001vortices,bao2013mathematical,wu2017regularized}. The total energy in the BEC problem is defined as
\begin{equation}
    E(\psi) = \int_{\mathbb{R}^n}\left[ \frac{1}{2}|\nabla \psi(x)|^2 + V(x) |\psi(x)|^2 + \frac{\beta}{2} |\psi(x)|^4 - \Omega \bar{\psi}(x)L_z(x)\right]{\rm d}x,
\end{equation}
where $x\in\mathbb{R}^d$ is the spatial coordinate vector, $\bar{\psi}$ denotes the complex conjugate of $\psi$, $L_z = -i(x\partial_y - y\partial_x)$, $V(x)$ is an external trapping potential, $\Omega\in\mathbb{R}$ is an angular velocity, and $\beta$ is a given constant. Then, the ground state of a BEC is usually defined as the minimizer of the following nonconvex minimization problem
\begin{equation}
    \min_{\psi\in S} E(\psi),
\end{equation}
where $S$ is the spherical constraint and is defined as
\begin{equation*}
    S: = \left\{\psi~ | ~E(\psi) < \infty, \int_{\mathbb{R}^d}|\psi(x)|^2 dx = 1 \right\}.
\end{equation*}
By using a suitable discretization, such as finite differences or the sine pseudo-spectral and Fourier pseudo-spectral (FP) method, we can reformulate the BEC problem as follows:
\begin{equation}\label{pro:bec}
    \min_{x\in\mathbb{C}^M} \frac{1}{2}x^*Ax + \frac{\beta}{2}\sum_{i=1}^M|x_i|^4, ~~\mathrm{s.t.}~~ x\in S^M,
\end{equation}
where $S^M = \{x\in\mathbb{C}^M~|~\|x\|_2 = 1\}$ with a positive integer $M$ and $A\in\mathbb{C}^{M\times M}$ is a Hermitian matrix. We refer to \cite{wu2017regularized} for the details.

\begin{figure}[!htb]
\centering
% \subfigure{
% \includegraphics[width=0.45\textwidth]{fig/bec/1cpu_n_cmp_all.eps}}
% \subfigure{
% \includegraphics[width=0.45\textwidth]{fig/bec/1obj_n_cmp_all.eps}}
% \\
\subfigure{
\includegraphics[width=0.45\textwidth]{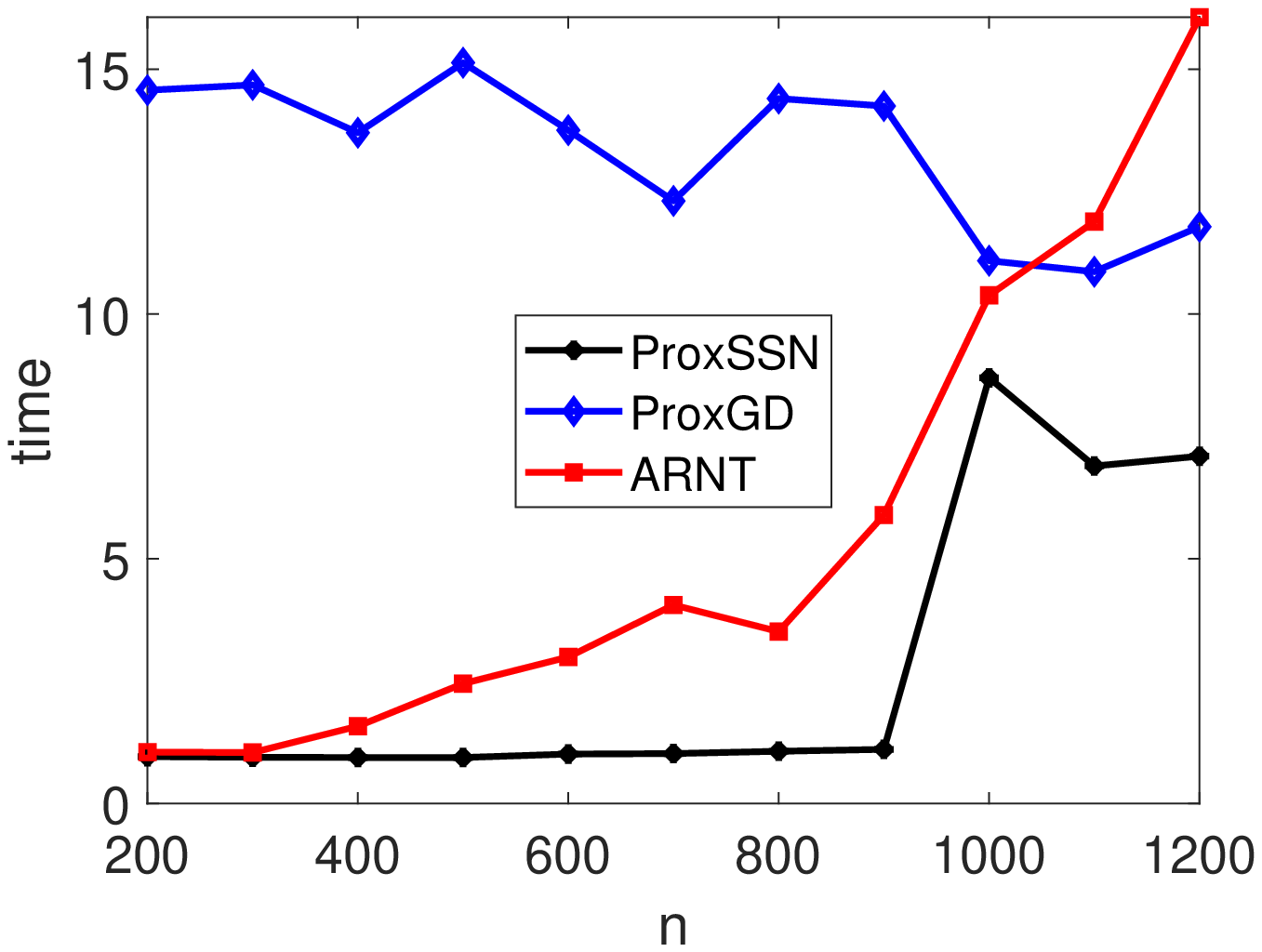}}
\subfigure{
\includegraphics[width=0.45\textwidth]{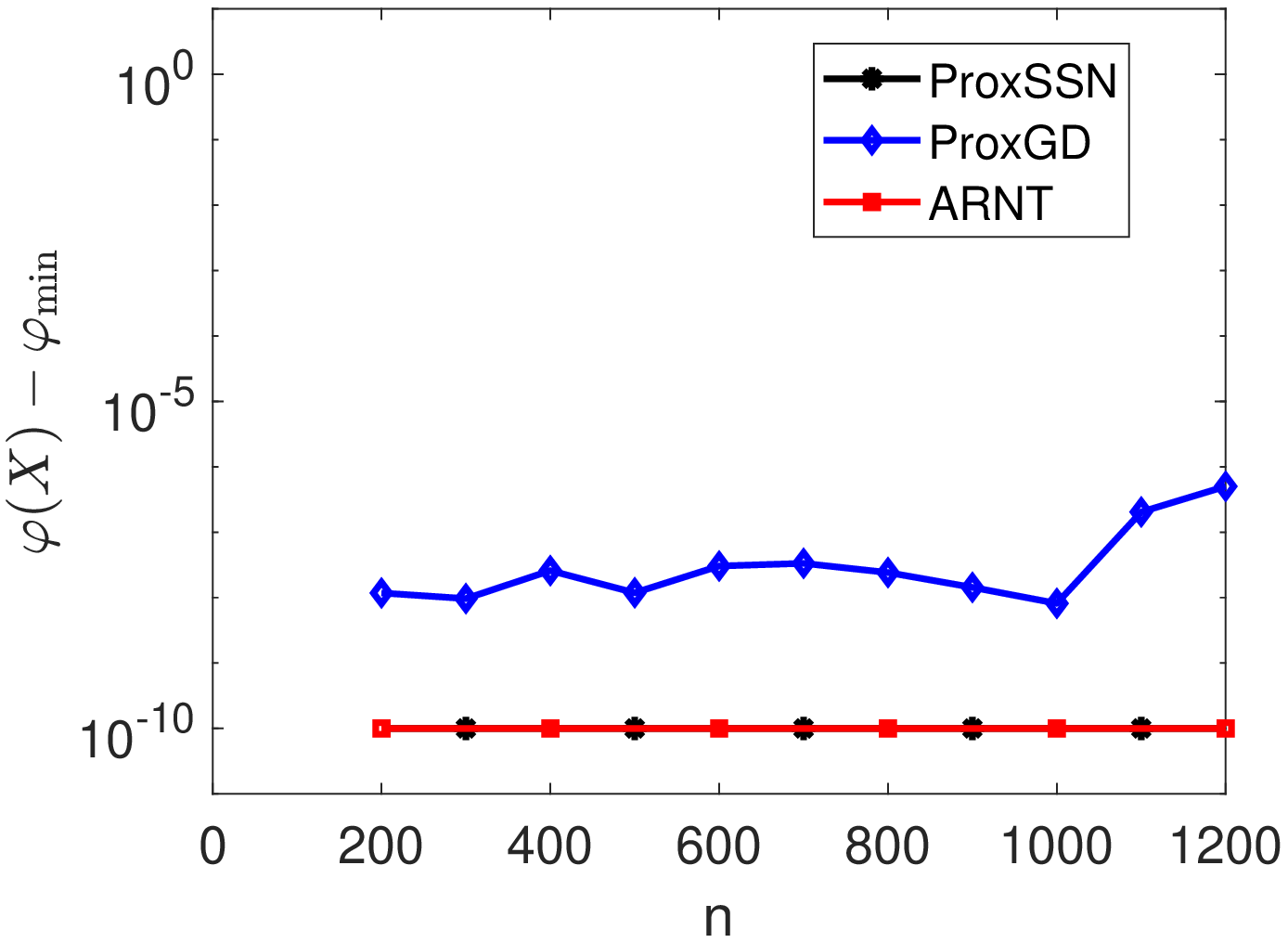}}
\caption{Comparisons of wall-clock time and the objective function values on the BEC problem \eqref{pro:bec} with $\Omega = 0.2$ and ``b''.}\label{fig:perf_bec_m}
\end{figure}

The ProxGD and ProxSSN  are applied to problem \eqref{pro:bec} by setting
\begin{equation*}
    f(x): = \frac{1}{2}x^*Ax + \frac{\beta}{2}\sum_{i=1}^M|x_i|^4, ~~~h(x) =  \delta_{S^M}(x).
\end{equation*}
Since problem \eqref{pro:bec} can be seen as a smooth problem on the complex sphere, we do comparisons with the adaptive regularized Newton method (ARNT) in \cite{hu2018adaptive}.  All parameters of ProxGD and ProxSSN follow the setup discussed in subsection \ref{sec:spca} except $\mathrm{tol} = 10^{-6}$. The parameters of ARNT are the same as in \cite{hu2018adaptive},  we stop ARNT when the Riemannian gradient norm is less than $10^{-6}$ or the maximum number of iterations 500 is reached.  We take $d = 2$ and $V(x,y) = \frac{1}{2}x^2 + \frac{1}{2}y^2$. The BEC problem is discretized by FP on the bounded domain $(-16,16)^2$ with $\beta$ ranging from $500$ to $1000$ and $\Omega = 0,0.1,0.25$. Following the settings in \cite{wu2017regularized}, we use the mesh refinement procedure with the coarse meshes $(2^k+1) \times (2^k+1) (k=2,\cdots, 5)$ to gradually obtain an initial solution point on the finest mesh $(2^6+1) \times (2^6+1)$. all algorithms are tested with mesh refinement and start from the same initial point on the coarsest mesh with
\begin{equation*}
    \phi_a(x, y) = \frac{(1-\Omega)\phi_1(x,y)+\Omega\phi_2(x,y)}{\|(1-\Omega)\phi_1(x,y)+\Omega\phi_2(x,y)\|},~~\phi_b(x, y) = \frac{\phi_1(x,y)+\phi_2(x,y)}{\|\phi_1(x,y)+\phi_2(x,y)\|},
\end{equation*}
where $\phi_1(x,y) = \frac{1}{\sqrt{\pi}}e^{-(x^2+y^2)/2}$ and $\phi_2(x,y)= \frac{x+iy}{\sqrt{\pi}}e^{-(x^2+y^2)/2}$.
%\revise{add explainations on d/c2, siter}

\begin{figure}[!htb]
\centering
\subfigure{
\includegraphics[width=0.45\textwidth]{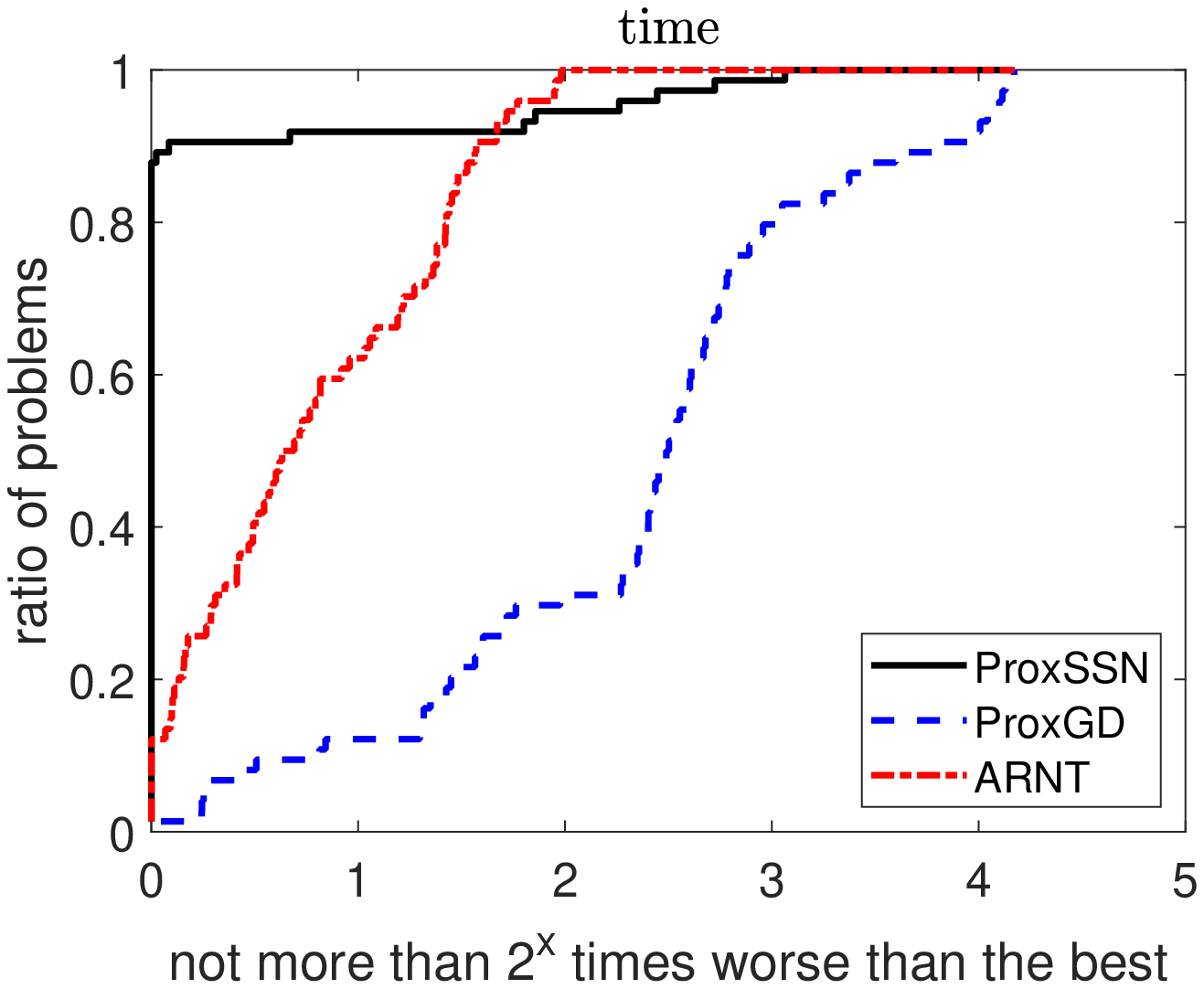}}
\subfigure{
\includegraphics[width=0.45\textwidth]{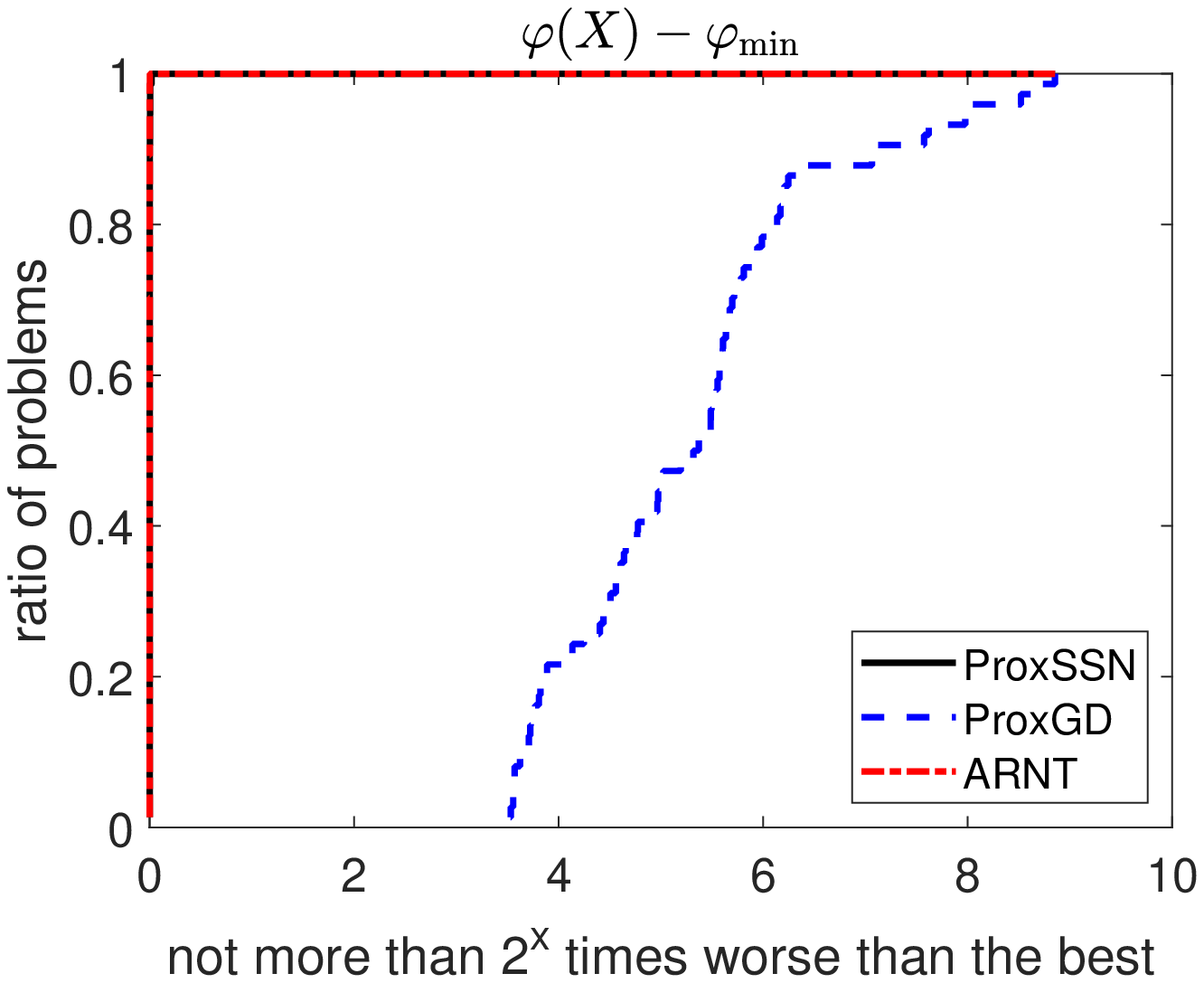}}
\caption{The performance profiles on the BEC problem \eqref{pro:bec}.}\label{fig:profile_bec_time}
\end{figure}
Table \ref{tab:lsr} gives detailed computational results. For the first column,``Initial'' denotes the type of the initial point, ``a'' and ``b'' are $\phi_a(x,y)$ and $\phi_b(x,y)$, respectively. For the iteration numbers in our table,  ``iter'' and ``siter'' denote the outer iterations and the average sub-iterations, respectively.  Note that ProxGD reaches the maximum iteration of 1000, which shows that ProxGD does not converge to the required accuracy in all cases. ProxSSN and ARNT find a point with almost the same objective function value,  while our algorithm ProxSSN is faster than ARNT in most cases. Figures \ref{fig:perf_bec_m} and \ref{fig:profile_bec_time} demonstrate the superiority of ProxSSN over ARNT and ProxGD.

\begin{scriptsize}
\setlength{\tabcolsep}{1.2pt}
%\footnotesize
\centering
\begin{longtable}{|c|ccc|ccc|ccc|ccc|}
\caption{Computational results of BEC }\label{tab:bec}\\
\hline
\multirow{2}{*}{$(\beta,\Omega,Initial)$} & \multicolumn{3}{c|}{ProxSSN} & \multicolumn{3}{c|}{ProxGD} & \multicolumn{3}{c|}{ARNT}     \\ \cline{2-10}
     & time   &  obj & iter (siter)  & time & obj & iter     & time & obj & iter (siter)     \\ \hline
\endfirsthead
\hline
\multirow{2}{*}{$(\beta,\Omega,Initial)$} & \multicolumn{3}{c|}{ProxSSN} & \multicolumn{3}{c|}{ProxGD} & \multicolumn{3}{c|}{ARNT}     \\ \cline{2-10}
     & time   &  obj & iter (siter)  & time & obj & iter     & time & obj & iter (siter)     \\ \hline
\endhead
\hline
\endfoot

 500 / 0.00 / a &  0.17 & 9.38492745 & 2 (46.0) &    10.75 & 9.38492745 & 1000 &   0.33 & 9.38492745 & 6 (17.3)\\ \hline
 500 / 0.10 / a &  0.94 & 9.38492744 & 3 (133.3) &    15.13 & 9.38492746 & 1000 &   2.45 & 9.38492744 & 7 (61.7)\\ \hline
 500 / 0.25 / a &  1.14 & 9.38492744 & 3 (133.3) &    12.65 & 9.38492747 & 1000 &   4.74 & 9.38492744 & 15 (73.3)\\ \hline
 500 / 0.00 / b &  0.74 & 9.38492745 & 3 (133.3) &    14.24 & 9.38492748 & 1000 &   1.21 & 9.38492745 & 7 (49.9)\\ \hline
 500 / 0.10 / b &  0.89 & 9.38492744 & 3 (133.3) &    14.39 & 9.38492746 & 1000 &   2.23 & 9.38492744 & 8 (60.3)\\ \hline
 500 / 0.25 / b &  0.98 & 9.38492744 & 3 (133.3) &    12.29 & 9.38492747 & 1000 &   3.50 & 9.38492744 & 11 (72.3)\\ \hline
 600 / 0.00 / a &  0.20 & 10.60175601 & 3 (95.0) &    11.38 & 10.60175602 & 1000 &   0.36 & 10.60175601 & 6 (17.2)\\ \hline
 600 / 0.10 / a &  1.01 & 10.60175601 & 3 (133.3) &    13.75 & 10.60175604 & 1000 &   2.99 & 10.60175601 & 8 (57.8)\\ \hline
 600 / 0.25 / a &  1.20 & 10.60175601 & 3 (133.3) &    12.75 & 10.60175606 & 1000 &   5.09 & 10.60175601 & 14 (70.3)\\ \hline
 600 / 0.00 / b &  0.99 & 10.60175601 & 3 (133.3) &    16.95 & 10.60175602 & 1000 &   1.75 & 10.60175601 & 6 (54.0)\\ \hline
 600 / 0.10 / b &  1.07 & 10.60175601 & 3 (133.3) &    14.52 & 10.60175604 & 1000 &   3.64 & 10.60175601 & 11 (52.4)\\ \hline
 600 / 0.25 / b &  1.10 & 10.60175601 & 3 (133.3) &    12.24 & 10.60175606 & 1000 &   4.61 & 10.60175601 & 10 (65.0)\\ \hline
 700 / 0.00 / a &  0.26 & 11.75508441 & 3 (95.0) &    11.90 & 11.75508441 & 1000 &   0.42 & 11.75508441 & 6 (15.2)\\ \hline
 700 / 0.10 / a &  1.02 & 11.75508441 & 3 (133.3) &    12.30 & 11.75508444 & 1000 &   4.05 & 11.75508441 & 6 (57.8)\\ \hline
 700 / 0.25 / a &  1.00 & 11.75508441 & 3 (133.3) &    11.52 & 11.75508453 & 1000 &   4.81 & 11.75508441 & 12 (60.5)\\ \hline
 700 / 0.00 / b &  0.87 & 11.75508441 & 3 (133.3) &    16.77 & 11.75508442 & 1000 &   1.78 & 11.75508441 & 6 (54.8)\\ \hline
 700 / 0.10 / b &  0.93 & 11.75508441 & 3 (133.3) &    11.83 & 11.75508444 & 1000 &   4.95 & 11.75508441 & 9 (47.8)\\ \hline
 700 / 0.25 / b &  1.12 & 11.75508441 & 3 (133.3) &    11.51 & 11.75508452 & 1000 &   4.93 & 11.75508441 & 10 (62.2)\\ \hline
 800 / 0.00 / a &  0.20 & 12.85654802 & 3 (95.3) &    11.09 & 12.85654802 & 1000 &   0.47 & 12.85654802 & 6 (15.0)\\ \hline
 800 / 0.10 / a &  1.06 & 12.85654802 & 3 (133.3) &    14.40 & 12.85654804 & 1000 &   3.51 & 12.85654802 & 12 (55.8)\\ \hline
 800 / 0.25 / a &  1.22 & 12.85654801 & 3 (133.3) &    12.85 & 12.85654804 & 1000 &   5.24 & 12.85654801 & 8 (62.9)\\ \hline
 800 / 0.00 / b &  0.77 & 12.85654802 & 3 (133.3) &    16.38 & 12.85654803 & 1000 &   1.71 & 12.85654802 & 6 (53.5)\\ \hline
 800 / 0.10 / b &  1.02 & 12.85654802 & 3 (133.3) &    14.80 & 12.85654804 & 1000 &   4.52 & 12.85654802 & 9 (49.2)\\ \hline
 800 / 0.25 / b &  1.13 & 12.85654801 & 3 (133.3) &    12.58 & 12.85654804 & 1000 &   6.63 & 12.85654801 & 14 (60.1)\\ \hline
 900 / 0.00 / a &  0.19 & 13.91448057 & 3 (91.3) &    10.89 & 13.91448057 & 1000 &   0.57 & 13.91448057 & 6 (16.0)\\ \hline
 900 / 0.10 / a &  1.10 & 13.91448057 & 3 (133.3) &    14.24 & 13.91448058 & 1000 &   5.89 & 13.91448057 & 14 (51.4)\\ \hline
 900 / 0.25 / a &  1.50 & 13.91448056 & 4 (150.0) &    14.54 & 13.91448058 & 1000 &   7.22 & 13.91448056 & 15 (64.2)\\ \hline
 900 / 0.00 / b &  0.53 & 13.91448057 & 2 (100.0) &    17.00 & 13.91448057 & 1000 &   2.21 & 13.91448057 & 6 (50.2)\\ \hline
 900 / 0.10 / b &  1.07 & 13.91448057 & 3 (133.3) &    15.21 & 13.91448058 & 1000 &   7.55 & 13.91448057 & 10 (53.8)\\ \hline
 900 / 0.25 / b &  1.16 & 13.91448056 & 3 (133.3) &    12.41 & 13.91448057 & 1000 &   8.21 & 13.91448056 & 11 (62.5)\\ \hline
 1000 / 0.00 / a &  0.23 & 14.93511997 & 2 (67.0) &    8.41 & 14.93511997 & 1000 &   0.90 & 14.93511997 & 6 (22.7)\\ \hline
 1000 / 0.10 / a &  8.69 & 14.93511995 & 4 (150.0) &    11.08 & 14.93511996 & 1000 &   10.38 & 14.93511995 & 14 (75.7)\\ \hline
 1000 / 0.25 / a &  4.90 & 14.93511986 & 5 (160.0) &    11.39 & 14.93512017 & 1000 &   13.75 & 14.93511986 & 21 (96.2)\\ \hline
 1000 / 0.00 / b &  2.60 & 14.93511997 & 3 (133.3) &    12.94 & 14.93511997 & 1000 &   3.98 & 14.93511997 & 10 (76.3)\\ \hline
 1000 / 0.10 / b &  9.34 & 14.93511995 & 4 (150.0) &    11.92 & 14.93511996 & 1000 &   12.44 & 14.93511995 & 13 (77.9)\\ \hline
 1000 / 0.25 / b &  2.43 & 14.93511986 & 5 (160.0) &    11.99 & 14.93512015 & 1000 &   17.62 & 14.93511986 & 18 (93.4)\\ \hline

\end{longtable}
\end{scriptsize}

\section{Conclusion}\label{sec:con}
This paper introduces a new concept of strong prox-regularity and validates it over many existing interesting applications, \revise{including composite optimization problems with weakly convex regularizer, smooth optimization problems on manifold, and several composite optimization problems on manifold.} Then a projected semismooth Newton method is proposed for solving a class of nonconvex optimization problems equipped with strong prox-regularity. The idea is to utilize the locally single-valued, Lipschitz continuous, and monotone properties of the residual mapping. The global convergence and local superlinear convergence results of the proposed algorithm are presented under standard conditions. Numerical results have convincingly demonstrated the effectiveness of our proposed method in various nonconvex composite problems, including the sparse PCA problem, the nonnegative PCA problem, the sparse least square regression, and the BEC problem.   

\section*{Acknowledgements} The authors are grateful to Prof.~Anthony Man-Cho So for his valuable comments and suggestions.

\bibliographystyle{siamplain}
\bibliography{ref}
\end{document}